\documentclass[12pt]{amsart}
\usepackage{amsmath}
\usepackage{amsxtra}
\usepackage{amscd}
\usepackage{amsthm}
\usepackage{amsfonts}
\usepackage{amssymb}
\usepackage{eucal}
\usepackage[shortlabels]{enumitem}
\usepackage{verbatim}
\usepackage[all]{xy}
\usepackage{mathrsfs}
\usepackage{tikz}
\usepackage{genyoungtabtikz} 
\usepackage{color}
\definecolor{deepjunglegreen}{rgb}{0.0, 0.29, 0.29}
\definecolor{darkspringgreen}{rgb}{0.09, 0.45, 0.27}
\definecolor{Red}{rgb}{0.7, 0,0}

\usepackage{stmaryrd}

\makeatletter
\usepackage{etex,etoolbox,needspace}
\pretocmd\section{\Needspace*{4\baselineskip}}{}{}
\makeatletter

\usepackage[pdftex,bookmarks=false,colorlinks=true,citecolor=darkspringgreen,debug=true,
  naturalnames=true,pdfnewwindow=true]{hyperref}

\newtheorem{thm}{Theorem}[subsection]

\newtheorem{cor}[thm]{Corollary}

\newtheorem{lem}[thm]{Lemma}
\newtheorem{prop}[thm]{Proposition}
\newtheorem{conj}[thm]{Conjecture}

\theoremstyle{definition}
\newtheorem{defn}[thm]{Definition}

\theoremstyle{remark}
\newtheorem{rem}[thm]{Remark}

\newcommand{\nc}{\newcommand}
\nc{\renc}{\renewcommand} \nc{\ssec}{\subsection}
\nc{\sssec}{\subsubsection}

\nc{\on}{\operatorname} \nc{\wh}{\widehat}
\nc\ol{\overline} \nc\ul{\underline} \nc\wt{\widetilde}

\newcommand{\red}[1]{{\color{Red}#1}}

\nc{\BA}{{\mathbb{A}}} \nc{\BC}{{\mathbb{C}}} \nc{\BF}{{\mathbb{F}}}
\nc{\BD}{{\mathbb{D}}} \nc{\BG}{{\mathbb{G}}} \nc{\BQ}{{\mathbb{Q}}}
\nc{\BM}{{\mathbb{M}}} \nc{\BN}{{\mathbb{N}}} \nc{\BO}{{\mathbb{O}}}
\nc{\BP}{{\mathbb{P}}} \nc{\BR}{{\mathbb{R}}}
\nc{\BZ}{{\mathbb{Z}}} \nc{\BS}{{\mathbb{S}}} \nc{\BW}{{\mathbb{W}}}

\nc{\CA}{{\mathcal{A}}} \nc{\CB}{{\mathcal{B}}} \nc{\CalC}{{\mathcal{C}}} \nc{\CalD}{{\mathcal{D}}}
\nc{\CE}{{\mathcal{E}}} \nc{\CF}{{\mathcal{F}}}
\nc{\cF}{{\mathcal{F}}}
\nc{\CG}{{\mathcal{G}}} \nc{\CH}{{\mathcal{H}}}
\nc{\CI}{{\mathcal{I}}} \nc{\CK}{{\mathcal{K}}} \nc{\CL}{{\mathcal{L}}}
\nc{\CM}{{\mathcal{M}}} \nc{\CN}{{\mathcal{N}}}
\nc{\CO}{{\mathcal{O}}} \nc{\CP}{{\mathcal{P}}}
\nc{\cP}{{\mathcal{P}}}
\nc{\CQ}{{\mathcal{Q}}} \nc{\CR}{{\mathcal{R}}}
\nc{\CS}{{\mathcal{S}}} \nc{\CT}{{\mathcal{T}}}
\nc{\CU}{{\mathcal{U}}} \nc{\CV}{{\mathcal{V}}}  
\nc{\cV}{{\mathcal{V}}}  
\nc{\CY}{{\mathcal Y}}
\nc{\CW}{{\mathcal{W}}} \nc{\CZ}{{\mathcal{Z}}}

\nc{\scrL}{{\mathscr L}} \nc{\scrM}{{\mathscr M}} \nc{\scrQ}{{\mathscr Q}}
\nc{\scrX}{{\mathscr X}} \nc{\scrY}{{\mathscr Y}} \nc{\scrZ}{{\mathscr Z}}

\nc{\cM}{{\check{\mathcal M}}{}} \nc{\csM}{{\check{\mathcal A}}{}}
\nc{\oM}{{\overset{\circ}{\mathcal M}}{}}
\nc{\obM}{{\overset{\circ}{\mathbf M}}{}}
\nc{\oCA}{{\overset{\circ}{\mathcal A}}{}}
\nc{\obA}{{\overset{\circ}{\mathbf A}}{}}
\nc{\ooM}{{\overset{\circ}{M}}{}}
\nc{\osM}{{\overset{\circ}{\mathsf M}}{}}
\nc{\vM}{{\overset{\bullet}{\mathcal M}}{}}
\nc{\nM}{{\underset{\bullet}{\mathcal M}}{}}
\nc{\oD}{{\overset{\circ}{\mathcal D}}{}}
\nc{\obD}{{\overset{\circ}{\mathbf D}}{}}
\nc{\oA}{{\overset{\circ}{\mathbb A}}{}}
\nc{\op}{{\overset{\bullet}{\mathbf p}}{}}
\nc{\cp}{{\overset{\circ}{\mathbf p}}{}}
\nc{\oU}{{\overset{\bullet}{\mathcal U}}{}}
\nc{\ofZ}{{\overset{\circ}{\mathfrak Z}}{}}

\nc{\fa}{{\mathfrak{a}}} \nc{\fb}{{\mathfrak{b}}}
\nc{\fd}{{\mathfrak{d}}} \nc{\fe}{{\mathfrak{e}}} \nc{\ff}{{\mathfrak{f}}}
\nc{\fg}{{\mathfrak{g}}} \nc{\fgl}{{\mathfrak{gl}}}
\nc{\fh}{{\mathfrak{h}}} \nc{\fri}{{\mathfrak{i}}}
\nc{\fj}{{\mathfrak{j}}} \nc{\fk}{{\mathfrak{k}}} \nc{\fl}{{\mathfrak{l}}}
\nc{\fm}{{\mathfrak{m}}} \nc{\fn}{{\mathfrak{n}}}
\nc{\ft}{{\mathfrak{t}}} \nc{\fu}{{\mathfrak{u}}} \nc{\fv}{{\mathfrak{v}}}
\nc{\fw}{{\mathfrak{w}}} \nc{\fz}{{\mathfrak{z}}}
\nc{\fp}{{\mathfrak{p}}} \nc{\fq}{{\mathfrak{q}}} \nc{\frr}{{\mathfrak{r}}}
\nc{\fs}{{\mathfrak{s}}} \nc{\fx}{{\mathfrak{x}}} 
\nc{\fsl}{{\mathfrak{sl}}}
\nc{\fso}{{\mathfrak{so}}} \nc{\fsp}{{\mathfrak{sp}}} \nc{\osp}{{\mathfrak{osp}}}
\nc{\hsl}{{\widehat{\mathfrak{sl}}}}
\nc{\hgl}{{\widehat{\mathfrak{gl}}}}
\nc{\hg}{{\widehat{\mathfrak{g}}}}
\nc{\chg}{{\widehat{\mathfrak{g}}}{}^\vee}
\nc{\hn}{{\widehat{\mathfrak{n}}}}
\nc{\chn}{{\widehat{\mathfrak{n}}}{}^\vee}

\nc{\fA}{{\mathfrak{A}}} \nc{\fB}{{\mathfrak{B}}} \nc{\fC}{{\mathfrak{C}}}
\nc{\fD}{{\mathfrak{D}}} \nc{\fE}{{\mathfrak{E}}}
\nc{\fF}{{\mathfrak{F}}} \nc{\fG}{{\mathfrak{G}}} \nc{\fH}{{\mathfrak{H}}}
\nc{\fI}{{\mathfrak{I}}} \nc{\fJ}{{\mathfrak{J}}}
\nc{\fK}{{\mathfrak{K}}} \nc{\fL}{{\mathfrak{L}}}
\nc{\fM}{{\mathfrak{M}}} \nc{\fN}{{\mathfrak{N}}}
\nc{\frP}{{\mathfrak{P}}} \nc{\fQ}{{\mathfrak{Q}}} \nc{\fR}{{\mathfrak{R}}}
\nc{\fS}{{\mathfrak{S}}} \nc{\fT}{{\mathfrak{T}}} \nc{\fU}{{\mathfrak{U}}}
\nc{\fV}{{\mathfrak{V}}} \nc{\fW}{{\mathfrak{W}}}
\nc{\fX}{{\mathfrak{X}}} \nc{\fY}{{\mathfrak{Y}}}
\nc{\fZ}{{\mathfrak{Z}}}

\nc{\ba}{{\mathbf{a}}}
\nc{\bb}{{\mathbf{b}}} \nc{\bc}{{\mathbf{c}}} \nc{\be}{{\mathbf{e}}}
\nc{\bg}{{\mathbf{g}}} \nc{\bj}{{\mathbf{j}}} \nc{\bm}{{\mathbf{m}}}
\nc{\bn}{{\mathbf{n}}} \nc{\bp}{{\mathbf{p}}}
\nc{\bq}{{\mathbf{q}}} \nc{\br}{{\mathbf{r}}} \nc{\bt}{{\mathbf{t}}}
\nc{\bfu}{{\mathbf{u}}} \nc{\bv}{{\mathbf{v}}}
\nc{\bx}{{\mathbf{x}}} \nc{\by}{{\mathbf{y}}} \nc{\bz}{{\mathbf{z}}}
\nc{\bw}{{\mathbf{w}}} \nc{\bA}{{\mathbf{A}}}
\nc{\bB}{{\mathbf{B}}} \nc{\bC}{{\mathbf{C}}}
\nc{\bD}{{\mathbf{D}}} \nc{\bF}{{\mathbf{F}}} \nc{\bG}{{\mathbf{G}}}
\nc{\bH}{{\mathbf{H}}} \nc{\bI}{{\mathbf{I}}} \nc{\bJ}{{\mathbf{J}}}
\nc{\bK}{{\mathbf{K}}} \nc{\bM}{{\mathbf{M}}} \nc{\bN}{{\mathbf{N}}}
\nc{\bO}{{\mathbf{O}}} \nc{\bS}{{\mathbf{S}}} \nc{\bT}{{\mathbf{T}}}
\nc{\bU}{{\mathbf{U}}} \nc{\bV}{{\mathbf{V}}} \nc{\bW}{{\mathbf{W}}}
\nc{\bX}{{\mathbf{X}}}
\nc{\bY}{{\mathbf{Y}}} \nc{\bP}{{\mathbf{P}}}
\nc{\bZ}{{\mathbf{Z}}} \nc{\bh}{{\mathbf{h}}}

\nc{\sA}{{\mathsf{A}}} \nc{\sB}{{\mathsf{B}}}
\nc{\sC}{{\mathsf{C}}} \nc{\sD}{{\mathsf{D}}}
\nc{\sE}{{\mathsf{E}}} \nc{\sF}{{\mathsf{F}}} \nc{\sG}{{\mathsf{G}}} \nc{\sH}{{\mathsf{H}}}
\nc{\sI}{{\mathsf{I}}} \nc{\sK}{{\mathsf{K}}} \nc{\sL}{{\mathsf{L}}}
\nc{\sfm}{{\mathsf{m}}} \nc{\sM}{{\mathsf{M}}} \nc{\sN}{{\mathsf{N}}}
\nc{\sO}{{\mathsf{O}}} \nc{\sQ}{{\mathsf{Q}}} \nc{\sP}{{\mathsf{P}}}
\nc{\sT}{{\mathsf{T}}} \nc{\sZ}{{\mathsf{Z}}}
\nc{\sV}{{\mathsf{V}}} \nc{\sW}{{\mathsf{W}}} \nc{\sn}{{\mathsf{n}}}
\nc{\sfp}{{\mathsf{p}}} \nc{\sq}{{\mathsf{q}}} \nc{\sr}{{\mathsf{r}}}
\nc{\sfs}{{\mathsf{s}}} \nc{\st}{{\mathsf{t}}} \nc{\sfb}{{\mathsf{b}}}
\nc{\sfc}{{\mathsf{c}}} \nc{\sd}{{\mathsf{d}}} \nc{\sg}{{\mathsf{g}}}
\nc{\sz}{{\mathsf{z}}}

\nc{\tA}{{\widetilde{\mathbf{A}}}}
\nc{\tB}{{\widetilde{\mathcal{B}}}}
\nc{\tg}{{\widetilde{\mathfrak{g}}}} \nc{\tG}{{\widetilde{G}}}
\nc{\TM}{{\widetilde{\mathbb{M}}}{}}
\nc{\tO}{{\widetilde{\mathsf{O}}}{}}
\nc{\tU}{{\widetilde{\mathfrak{U}}}{}} \nc{\TZ}{{\tilde{Z}}}
\nc{\tx}{{\tilde{x}}} \nc{\tbv}{{\tilde{\bv}}}
\nc{\tfP}{{\widetilde{\mathfrak{P}}}{}} \nc{\tz}{{\tilde{\zeta}}}
\nc{\tmu}{{\tilde{\mu}}}

\nc{\urho}{\underline{\rho}} \nc{\uB}{\underline{B}}
\nc{\uC}{{\underline{\mathbb{C}}}} \nc{\ui}{\underline{i}}
\nc{\uj}{\underline{j}} \nc{\ofP}{{\overline{\mathfrak{P}}}}
\nc{\oB}{{\overline{\mathcal{B}}}}
\nc{\og}{{\overline{\mathfrak{g}}}} \nc{\oI}{{\overline{I}}}

\nc{\eps}{\varepsilon} \nc{\hrho}{{\hat{\rho}}} \nc{\balpha}{{\boldsymbol{\alpha}}}
\nc{\blambda}{{\boldsymbol{\lambda}}} \nc{\bmu}{{\boldsymbol{\mu}}} \nc{\bnu}{{\boldsymbol{\nu}}}
\nc{\btheta}{{\boldsymbol{\theta}}} \nc{\bzeta}{{\boldsymbol{\zeta}}} \nc{\bta}{{\boldsymbol{\eta}}}
\nc{\bbeta}{{\boldsymbol{\beta}}} \nc{\bkappa}{{\boldsymbol{\kappa}}} \nc{\bomega}{{\boldsymbol{\omega}}}

\nc{\one}{{\mathbf{1}}} \nc{\two}{{\mathbf{t}}}

\nc\DMO\DeclareMathOperator
\DMO\Sym{Sym}
\nc{\Tot}{{\mathop{\operatorname{\rm Tot}}}}
\nc{\Spec}{\mathop{\operatorname{\rm Spec}}}
\nc{\Ker}{\mathop{\operatorname{\rm Ker}}}
\nc{\Isom}{{\mathop{\operatorname{\rm Isom}}}}
\nc{\Hilb}{{\mathop{\operatorname{\rm Hilb}}}}
\nc{\deeq}{{\mathop{\operatorname{\rm deeq}}}}
\nc{\End}{{\mathop{\operatorname{\rm End}}}}
\nc{\Ran}{{\mathop{\operatorname{\rm Ran}}}}
\nc{\Ext}{{\mathop{\operatorname{\rm Ext}}}}
\nc{\Hom}{{\mathop{\operatorname{\rm Hom}}}}
\nc{\CHom}{{\mathop{\operatorname{{\mathcal{H}}\it om}}}}
\nc{\GL}{{\mathop{\operatorname{\rm GL}}}}
\nc{\PGL}{{\mathop{\operatorname{\rm PGL}}}}
\nc{\SL}{{\mathop{\operatorname{\rm SL}}}}
\nc{\SO}{{\mathop{\operatorname{\rm SO}}}}
\nc{\Sp}{{\mathop{\operatorname{\rm Sp}}}}
\nc{\PSp}{{\mathop{\operatorname{\rm PSp}}}}
\nc{\GSp}{{\mathop{\operatorname{\rm GSp}}}}
\nc{\Spin}{{\mathop{\operatorname{\rm Spin}}}}
\nc{\OSp}{{\mathop{\operatorname{\rm SOSp}}}}
\nc{\gr}{{\mathop{\operatorname{\rm gr}}}}
\nc{\Id}{{\mathop{\operatorname{\rm Id}}}}
\nc{\perf}{{\mathop{\operatorname{\rm perf}}}}
\nc{\defi}{{\mathop{\operatorname{\rm def}}}}
\nc{\length}{{\mathop{\operatorname{\rm length}}}}
\nc{\supp}{{\mathop{\operatorname{\rm supp}}}}
\nc{\HC}{{\mathcal H}{\mathcal C}}

\nc{\pr}{{\operatorname{pr}}}
\nc{\Cliff}{{\mathsf{Cliff}}}
\nc{\loc}{{\operatorname{loc}}} \nc{\lc}{{\operatorname{lc}}}
\nc{\Fl}{{\mathbf{Fl}}} \nc{\Ffl}{{\mathcal{F}\ell}}
\nc{\Fib}{{\mathsf{Fib}}}
\nc{\Coh}{{\mathsf{Coh}}} \nc{\FCoh}{{\mathsf{FCoh}}}
\nc{\Perf}{{\mathsf{Perf}}}
\nc{\wtimes}{\mathbin{\widetilde\times}}

\nc{\reg}{{\text{\rm reg}}} \nc{\ren}{{\text{\rm ren}}}
\nc{\self}{{\text{\rm self}}}
\nc{\gvee}{{\mathfrak g}^{\!\scriptscriptstyle\vee}}
\nc{\tvee}{{\mathfrak t}^{\!\scriptscriptstyle\vee}}
\nc{\nvee}{{\mathfrak n}^{\!\scriptscriptstyle\vee}}
\nc{\bvee}{{\mathfrak b}^{\!\scriptscriptstyle\vee}}
       \nc{\rhovee}{\rho^{\!\scriptscriptstyle\vee}}

\nc{\cplus}{{\mathbf{C}_+}} \nc{\cminus}{{\mathbf{C}_-}}
\nc{\cthree}{{\mathbf{C}_*}} \nc{\Qbar}{{\bar{Q}}}

\newcommand\iso{\mathbin{\vphantom{j^{X^2}}\smash{\overset{\sim}{\vphantom{\rule{0pt}{0.20em}}\smash{\longrightarrow}}}}}
\nc{\Gtimes}{\vphantom{j^{X^2}}\smash{\overset{G}{\vphantom{\rule{0pt}{0.30em}}\smash{\times}}}}
\nc{\sGtimes}{\vphantom{j^{X^2}}\smash{\overset{\mathsf G}{\vphantom{\rule{0pt}{0.30em}}\smash{\times}}}}

\nc{\svee}{{\!\scriptscriptstyle\vee}}

\nc{\bOmega}{{\overline{\Omega}}}

\nc{\seq}[1]{\stackrel{#1}{\sim}}

\nc{\aff}{{\operatorname{aff}}}
\nc{\fin}{{\operatorname{fin}}}
\nc{\mir}{{\operatorname{mir}}}
\nc{\triv}{{\operatorname{triv}}}
\nc{\ext}{{\operatorname{ext}}}
\nc{\righ}{{\operatorname{right}}}
\nc{\lef}{{\operatorname{left}}}
\nc{\forg}{{\operatorname{forg}}}
\nc{\fid}{{\operatorname{fd}}}
\nc{\odd}{{\operatorname{odd}}}
\nc{\even}{{\operatorname{even}}}
\nc{\modu}{{\operatorname{-mod}}}
\nc{\Gr}{{\operatorname{Gr}}}
\nc{\FT}{{\operatorname{FT}}}
\nc{\Mat}{{\operatorname{Mat}}}
\nc{\MSt}{{\operatorname{MSt}}}
\nc{\sph}{{\operatorname{sph}}}
\nc{\GR}{{\mathbf{Gr}}}
\nc{\Perv}{{\operatorname{Perv}}}
\nc{\Rep}{{\operatorname{Rep}}}
\nc{\Ind}{{\operatorname{Ind}}}
\nc{\IC}{{\operatorname{IC}}}
\nc{\Bun}{{\operatorname{Bun}}}
\nc{\Proj}{{\operatorname{Proj}}}
\nc{\Stab}{{\operatorname{Stab}}}
\nc{\pt}{{\operatorname{pt}}}
\nc{\bfmu}{{\boldsymbol{\mu}}}
\nc{\bfomega}{{\boldsymbol{\omega}}}
\nc{\calM}{\mathcal M}
\nc{\calA}{\mathcal A}
\nc{\calO}{\mathcal O}
\nc{\CC}{\mathcal C}
\nc{\cC}{\mathcal C}
\nc{\calN}{\mathcal N}
\nc{\grg}{\mathfrak g}
\nc{\dslash}{/\!\!/}
\nc{\tslash}{/\!\!/\!\!/}
\nc\grt{\mathfrak t}
\nc\bfM{\mathbf M}
\nc\bfN{\mathbf N}
\nc\Sig{\Sigma}
\nc\ZZ{\mathbb{Z}}
\nc\calC{\mathcal C}
\nc\calF{\mathcal F}
\nc\calX{\mathcal X}
\nc\calY{\mathcal Y}
\nc\QCoh{\operatorname{QCoh}}
\nc\Sch{\operatorname{Sch}}
\nc\IndCoh{\operatorname{IndCoh}}
\nc\Maps{\operatorname{Maps}}
\nc\Dmod{D-\operatorname{mod}}
\newcommand\Hecke{\operatorname{Hecke}}
\nc{\calD}{\mathcal D}
\nc\bfO{\mathbf O}

\nc\GG{\mathbb G}
\nc\calK{\mathcal K}
\nc{\calG}{\mathcal G}
\nc\RHom{\operatorname{RHom}}
\nc\Res{\operatorname{Res}}
\nc\Av{\operatorname{Av}}
\nc{\RH}{{\operatorname{RH}}}
\nc{\RT}{{\operatorname{RT}}}
\nc{\DR}{{\operatorname{DR}}}

\nc\grs{\mathfrak s}
\nc{\tilX}{\widetilde X}
\nc\calB{\mathcal B}
\nc\calS{\mathcal S}
\nc\calT{\mathcal T}
\nc\calZ{\mathcal Z}
\nc\LS{\operatorname{LocSys}}
\nc\bfL{\on{\mathbf L}}

\newcommand*\circled[1]
{*{#1}}
\makeatletter
\newcommand{\raisemath}[1]{\mathpalette{\raisem@th{#1}}}
\newcommand{\raisem@th}[3]{\raisebox{#1}{$#2#3$}}
\nc{\binlim}[2][]{\def\@tempa{#1}\@ifnextchar^{\@binlim{#2}}{\@binlim{#2}^{}}}
\def\@binlim#1^#2{\mathbin{\@ifempty{#2}{\mathop{#1}}{\mathop{#1}\@xp\displaylimits\@tempa^{#2}}}}

\makeatother

\nc\cX{{\mathcal X}}

\nc\Gm{{\mathbb G_m}}
\nc{\isoto}
    {\stackrel{\displaystyle\raisemath{-.2ex}{\sim}}\to}
\nc\cD\CalD
\nc\setm\setminus
\nc\cE\CE
\renc\Hecke{\mathit{\CH\kern-.2ex ecke}}
\nc\bsl\backslash
\nc\Fq{\mathbb F_q}
\nc\x\times
\nc\bGO{{\bG_\bO}}
\nc\bla\blambda
\nc\blt\bullet
\nc\opp{{\on{op}}}
\nc\srel\stackrel
\nc\tbx{\binlim{\widetilde\boxtimes{}}}
\nc\into\hookrightarrow
\nc\otto\leftrightarrow
\renc\setminus\smallsetminus
\nc\phitau{\varphi\tau}
\newenvironment{i-ii-iii}{%
\begin{enumerate}
}%
{\end{enumerate}}
\nc\ceil[1]{\lceil#1\rceil}  \nc\floor[1]{\lfloor#1\rfloor}
\nc\Lie{\on{Lie}}
\nc\sS{{\mathsf S}}
\nc\vvv{\ensuremath{\red\surd}}
\nc\la\lambda
\def\arxiv#1{\href{http://arxiv.org/abs/#1}{\tt arXiv:#1}} 
\nc\kap{\kappa}
\nc\gra{\mathfrak a}
\nc\diag{\mathrm{diag}}
\nc\gl{\mathfrak{gl}}
\nc\sTr{\operatorname{sTr}}
\nc\hatG{\widehat{G}}
\nc\calL{\mathcal L}
\nc\Whit{\operatorname{Whit}}
\nc\KL{\operatorname{KL}}

\makeatletter
\renewcommand{\subsection}{\@startsection{subsection}{2}{0pt}{-3ex
plus -1ex minus -0.2ex}{-2mm plus -0pt minus
-2pt}{\normalfont\bfseries}} \makeatother

\numberwithin{equation}{subsection}
\allowdisplaybreaks
\nc\mto{\mapsto }
\nc\en{\enspace }

\begin{document}

\author[M.Finkelberg]{Michael Finkelberg}
\address{Einstein Institute of Mathematics, The Hebrew University of Jerusalem,
  Edmond J. Safra Campus, Giv’at Ram, Jerusalem, 91904, Israel;
\newline  National Research University Higher School of Economics;
\newline Skolkovo Institute of Science and Technology}
\email{fnklberg@gmail.com}

\author[V.Ginzburg]{Victor Ginzburg}
\address{Department of Mathematics, University of Chicago, Chicago, USA}
\email{vityaginzburg@gmail.com}

\author[R.Travkin]{Roman Travkin}
\address{Skolkovo Institute of Science and Technology, Moscow, Russia}
\email{roman.travkin2012@gmail.com}

\title
{Lagrangian subvarieties of hyperspherical varieties}
\dedicatory{To our friend Roman Bezrukavnikov on his 50th birthday}



\begin{abstract}
  Given a hyperspherical $G$-variety $\scrX$ we consider the zero moment level $\Lambda_\scrX\subset\scrX$
  of the action of a Borel subgroup $B\subset G$. We conjecture that $\Lambda_\scrX$ is Lagrangian.
  For the dual $G^\vee$-variety $\scrX^\vee$, we conjecture that that there is a bijection between
  the sets of irreducible components $\on{Irr}\Lambda_\scrX$ and $\on{Irr}\Lambda_{\scrX^\vee}$. We check
  this conjecture for all the hyperspherical equivariant slices, and for all the basic classical
  Lie superalgebras.
\end{abstract}

\maketitle

\tableofcontents

\section{Introduction}

\subsection{Dual hyperspherical varieties}
Recall that a complex algebraic symplectic variety $\scrX$ equipped with a hamiltonian action of
a complex reductive group $G$ is called {\em coisotropic} if general $G$-orbits are coisotropic;
equivalently if the field of $G$-invariant rational functions $\BC(\scrX)^G$ is Poisson-commutative,
see e.g.~\cite{v}. The main source of coisotropic $G$-varieties consists of cotangent bundles of
smooth spherical $G$-varieties.

D.~Ben-Zvi, Y.~Sakellaridis and A.~Venkatesh specify in~\cite[\S3.5.1]{bzsv} a subclass of
{\em hyperspherical} coisotropic $G$-varieties. Under certain conditions~\cite[\S4]{bzsv},
for a hyperspherical $G$-variety $\scrX$, the authors of {\em loc.cit.} define the dual hyperspherical $G^\vee$-variety $\scrX^\vee$.
Here $G^\vee$ stands for the Langlands dual group of $G$. In~\cite[Conjecture 7.5.1]{bzsv}, the authors propose an equivalence of
the category of $G[\![t]\!]$-equivariant modules over a quantization of the loop space of $\scrX$ with
the category of $G^\vee$-equivariant coherent sheaves over a {\em shearing} of $\scrX^\vee$.

We propose a much more elementary relation between $G\circlearrowright\scrX$ and
$G^\vee\circlearrowright\scrX^\vee$. Specifically, we choose a Borel subgroup $B\subset G$ with Lie algebra
$\fb\subset\fg$, and let $\fb^\perp\subset\fg^*$ be its annihilator. We denote by
$\bmu\colon\scrX\to\fg^*$ the moment map, and we denote by $\Lambda_\scrX$ the preimage
$\bmu^{-1}(\fb^\perp)$. More invariantly, we consider the flag variety $\CB=G/B$ with cotangent bundle
$T^*\CB$, and we denote by $\widetilde{\Lambda}_\scrX$ the zero fiber of the moment map
$\tilde\bmu\colon\scrX\times T^*\CB\to\fg^*$ of the diagonal $G$-action.

More generally, for $\lambda\in\fg^*$ such that $\bar\lambda:=\lambda|_\fb$
is a character of $\fb$, we consider the subvariety
$\bmu^{-1}(\lambda+\fb^\perp)$. According to~Proposition~\ref{prop equiv} below,
$\scrX$ is coisotropic iff for any such $\lambda$ in a nonempty open $G$-invariant
subvariety $Z\subset\bmu(\scrX)$, the intersection
$\bmu^{-1}(\lambda+\fb^\perp)\cap\bmu^{-1}(Z)$ is Lagrangian in $\scrX$.
The first part of the following conjecture may be viewed as a limit 
$\lambda\to0$ version of~Proposition~\ref{prop equiv}.

\begin{conj}
  \label{weak}
  \textup{(a)} For a hyperspherical $G$-variety $\scrX$, each irreducible component
  of the subvariety $\Lambda_\scrX\subset\scrX$
  is Lagrangian. Equivalently, $\widetilde{\Lambda}_\scrX\subset\scrX\times T^*\CB$ is Lagrangian.

  \textup{(b)} The sets of irreducible components of Lagrangian subvarieties
  $\on{Irr}\Lambda_\scrX=\on{Irr}\widetilde{\Lambda}_\scrX$ and
  $\on{Irr}\Lambda_{\scrX^\vee}=\on{Irr}\widetilde{\Lambda}_{\scrX^\vee}$ have the same cardinality.
\end{conj}

Furthermore, we expect that there is a distinguished bijection between
$\on{Irr}\widetilde{\Lambda}_\scrX$ and $\on{Irr}\widetilde{\Lambda}_{\scrX^\vee}$.
Let $\CN_\fg\subset\fg$ be the nilpotent cone, and let $Z_\fg=T^*\CB\times_\CN T^*\CB$ be the Steinberg
variety of triples. Let $H(Z_\fg)$ be the top degree part of the Borel-Moore homology of $Z_\fg$ (of degree
equal to $2\dim(Z_\fg)$). Then $H(Z_\fg)$ equipped with the convolution operation forms an algebra
isomorphic to the group algebra $\BC[W_\fg]$ of the Weyl group of $\fg$, see
e.g.~\cite[Theorem 3.4.1]{cg}. A similar convolution operation gives rise to the action of
$H(Z_\fg)$ on $H(\widetilde{\Lambda}_\scrX)$: the homogeneous part of the Borel-Moore homology of  
$\widetilde{\Lambda}_\scrX$ of degree equal to $2\dim\widetilde{\Lambda}_\scrX$. This is a based
vector space with a basis of fundamental classes of irreducible components in
$\on{Irr}\widetilde{\Lambda}_\scrX$. Note that $\fg$ and $\fg^\svee$ share the same Weyl group, so
we have a canonical isomorphism $H(Z_\fg)\cong H(Z_{\fg^\vee})$.

\begin{conj}
  \label{medium}
  The 
  $H(Z_\fg)\cong H(Z_{\fg^\vee})$-modules $H(\widetilde{\Lambda}_\scrX)$
  and $H(\widetilde{\Lambda}_{\scrX^\vee})$ are isomorphic.
\end{conj}

Furthermore, let $Q\scrX$ and $Q\scrX^\vee$ denote quantizations of $\scrX$ and $\scrX^\vee$.
For example, if $\scrX\simeq T^*_\psi\scrY$ is polarizable (a twisted cotangent bundle), then
$Q\scrX$ is the ring of $\psi$-twisted differential operators on $\scrY$. We consider the (big)
derived category $\calC$ of $B$-equivariant objects in $Q\scrX\modu$, and denote by
$\calC^\scrX_{B\on{-equiv}}$ the category of compact objects in $\calC$.
Variant: we denote by $\widehat\calC{}^\scrX_{B\on{-equiv}}$ the category of locally compact objects
in $\calC$ (i.e.\ the ones that become compact after forgetting the equivariant structure).

We also consider the derived category $\calC^{\scrX^\vee}_{B^\vee\on{-mon,uni}}$ of $B^\vee$-monodromic
finitely generated $Q\scrX^\vee$-modules with unipotent monodromic.
Variant: $\widehat\calC{}^{\scrX^\vee}_{B^\vee\on{-mon,uni}}$
is the category of free prounipotent $B^\vee$-monodromic $Q\scrX^\vee$-modules (i.e.\ the ones whose
averaging $\on{Av}_{U^\vee}^{B^\vee}$ are finitely generated over $Q\scrX^\vee$).

We expect that $\calC^\scrX_{B\on{-equiv}},\ \widehat\calC{}^\scrX_{B\on{-equiv}},\
\calC^{\scrX^\vee}_{B^\vee\on{-mon,uni}}$ and $\widehat\calC{}^{\scrX^\vee}_{B^\vee\on{-mon,uni}}$
have {\em graded} versions (see e.g.~\cite[\S5.1]{bv})
$\calC^{\scrX,\on{gr}}_{B\on{-equiv}},\ \widehat\calC{}^{\scrX,\on{gr}}_{B\on{-equiv}},\
\calC^{\scrX^\vee,\on{gr}}_{B^\vee\on{-mon,uni}},\ \widehat\calC{}^{\scrX^\vee,\on{gr}}_{B^\vee\on{-mon,uni}}$ such that the following Koszul dualities hold:

\begin{conj}
  \label{strong}
  There are equivalences
  $\kappa\colon\calC^{\scrX,\on{gr}}_{B\on{-equiv}}\iso\calC^{\scrX^\vee,\on{gr}}_{B^\vee\on{-mon,uni}}$
  and $\widehat\kappa\colon\widehat\calC{}^{\scrX,\on{gr}}_{B\on{-equiv}}\iso
  \widehat\calC{}^{\scrX^\vee,\on{gr}}_{B^\vee\on{-mon,uni}}$, cf.~\cite{by}.
\end{conj}

\begin{rem}
  Conjecture~\ref{strong} or at least its 2-periodic (ungraded) version, is likely to follow
  from the loop rotation equivariant version of the main local
  conjecture~\cite[Conjecture 7.5.1]{bzsv} by passing to the fixed points of the
  loop rotations. More precisely, we fix a dominant regular coweight $\lambda$
  and consider the $\BG_m$-action on $\scrX(\!(t)\!)$ via $(t,t^\lambda)$,
  where the first factor acts via loop rotations, and the second one via the natural
  $G$-action. Then by the localization theorem, in the polarizable case $\scrX=T^*_\psi\scrY$,
  the localization of $Q\scrX(\!(t)\!)\on{-mod}^{G[\![t]\!]}$ is
  the category of $\psi$-twisted $D$-modules on the stack of $\BG_m$-fixed
  points of
  $G[\![t]\!]\backslash\scrY(\!(t)\!)$. This fixed point stack is $B\backslash\scrY$.
  In the coherent side of~\cite[Conjecture 7.5.1]{bzsv}, after localization of
  weakly $G^\vee$-equivariant sheared asymptotic $Q\scrX^\vee$-modules
  we obtain strongly $G^\vee$-equivariant
  $Q\scrX^\vee\otimes\CalD_{\hat\lambda}(G^\vee/B^\vee)$-modules. The latter category
  is equivalent to $\calC^{\scrX^\vee}_{B^\vee\on{-mon,uni}}$.
  \end{rem}

\subsection{Basic classical Lie superalgebras and equivariant slices}
In this note we check~Conjecture~\ref{weak} for a certain class of hyperspherical varieties
$G\circlearrowright\scrX$ considered in~\cite{fu}: the hyperspherical equivariant slices.
The dual class $G^\vee\circlearrowright\scrX^\vee$ consists of linear symplectic representations arising
from (the odd parts of) the basic classical Lie superalgebras.\footnote{Strictly speaking,
we check a certain generalization of~Conjecture~\ref{weak} since for certain cases within our
class of hyperspherical varieties, the action $G\circlearrowright\scrX$ has {\em anomaly}, and
these cases are not formally covered by~\cite[\S4]{bzsv}. In these cases $G^\vee$ is not the
Langlands dual of $G$, but rather a certain {\em metaplectic} Langlands dual, see e.g.~\cite{gl}.} 
Namely, part~\ref{weak}(a) is
immediate since all the hyperspherical varieties in question turn out to be {\em pseudospherical},
that is isomorphic to twisted cotangent bundles of certain $B$-varieties with finitely many
{\em relevant} $B$-orbits, cf.~Remark~\ref{super}.

Part~\ref{weak}(b) is checked case by case. In case
$G^\vee\circlearrowright\scrX^\vee=\GL(M)\times\GL(N)\circlearrowright T^*\Hom(\BC^N,\BC^M)$ (the odd
part of $\sg=\fgl(M|N)$), both $\on{Irr}\Lambda_\scrX$ and $\on{Irr}\Lambda_{\scrX^\vee}$
are indexed by the set $\fR_{M,N}$ of placements of non-attacking rooks on the $M\times N$-chessboard,
see~Proposition~\ref{gl(M|N) lagr}. We also introduce a combinatorial notion of a {\em maze},
see~Definition~\ref{maze} and~Fig.~\ref{a maze}. The set $\fM_{M,N}$ of $M\times N$-mazes is in a natural bijection
with $\fR_{M,N}\simeq\on{Irr}\Lambda_\scrX\simeq\on{Irr}\Lambda_{\scrX^\vee}$.

In case $G^\vee\circlearrowright\scrX^\vee=\SO(m)\times\Sp(2n)\circlearrowright\BC^m_+\otimes\BC^{2n}_-$
(the odd part of $\sg=\mathfrak{osp}(m|2n)$), both
$\on{Irr}\Lambda_\scrX$ and $\on{Irr}\Lambda_{\scrX^\vee}$ are indexed by the set $\fM_{m,2n}^\iota$
of centrally symmetric $m\times2n$-mazes, see~Proposition~\ref{osp lagr}. In particular,
if $m=2n$ (resp.\ $m=2n+2$),
$\sharp\fM_{2n,2n}^\iota=\sum_{k=1}^{n+1}\frac{n!}{(k-1)!}\binom{n}{k-1}2^{n-k+1}$
(resp.\ $\sharp\fM_{2n+2,2n}^\iota=\sum_{k=1}^{n+1}\frac{(n+1)!}{k!}\binom{n}{k-1}2^{n-k+1}$)
is the number of signed partitions into lists, see~\cite[Propositions 3.16, 3.14]{ce}.

In case $G^\vee\circlearrowright\scrX^\vee=\Spin(7)\times\Sp(2)\circlearrowright\BC^8_+\otimes\BC^2_-$
(the odd part of the exceptional Lie superalgebra $\sg=\mathfrak{f}(4)$), both
$\on{Irr}\Lambda_\scrX$ and $\on{Irr}\Lambda_{\scrX^\vee}$ have cardinality~9, see~Proposition~\ref{f4 lagr}.
Finally, in case
$G^\vee\circlearrowright\scrX^\vee=\on{G}_2\times\Sp(2)\circlearrowright\BC^7_+\otimes\BC^2_-$
(the odd part of the exceptional Lie superalgebra $\sg=\mathfrak{g}(3)$), both
$\on{Irr}\Lambda_\scrX$ and $\on{Irr}\Lambda_{\scrX^\vee}$ have cardinality~7, see~\cite{k}.

We collect the results of counting $\sharp\on{Irr}\Lambda_{\scrX}=\sharp\on{Irr}\Lambda_{\scrX^\vee}$
in~Table~\ref{tab}. In the first column $\CS_{(a,1^b)}$ stands for a Slodowy slice to a nilpotent
orbit of the Jordan `hook' type $(a,1^b)$, and $\CS_{\on{short}}$ in the last row stands for a
Slodowy slice to the 8-dimensional nilpotent orbit containing the short root vectors in the exceptional
Lie algebra $\fg_2$. Also note that in the rows~4,5,6,10 the group $G^\vee$ is not the Langlands dual
of $G$, but rather a certain {\em metaplectic} Langlands dual (due to the presence of a certain
{\em anomaly} of the action $G\circlearrowright\scrX$).

\begin{table}[h]
  \centering
  \begin{tabular}{c|c|c|c}
    equivariant  & symplectic repre- &
    Lie super- & cardinality \\
    slice $G\circlearrowright\scrX$ & sentation $G^\vee\circlearrowright\scrX^\vee$ &
    algebra $\sg$ & $\sharp\on{Irr}\Lambda_\scrX=$\\
    & & & $\sharp\on{Irr}\Lambda_{\scrX^\vee}$\\
    \hline
    $\GL(N)\times\GL(N)\circlearrowright$  &
    $\GL(N)\times\GL(N)$  &
    $\fgl(N|N)$ & $\sharp\fM_{N,N}$\\
    $T^*(\GL(N)\times\BC^N)$ & $\circlearrowright T^*\Hom(\BC^N,\BC^N)$ & & \\
    \hline
$\GL(N)\times\GL(M)\circlearrowright$  &
    $\GL(N)\times\GL(M)$  &
    $\fgl(M|N)$ & $\sharp\fM_{M,N}$\\
    $\GL(N)\times\CS_{(N-M,1^M)}$ & $\circlearrowright T^*\Hom(\BC^M,\BC^N)$ & $M<N$ & \\
    \hline
$\Sp(2n)\times\Sp(2n)\circlearrowright$  &
    $\SO(2n+1)\times\Sp(2n)$  &
    $\osp(2n+1|2n)$ & $\sharp\fM^\iota_{2n,2n}=$\\
    $(T^*\Sp(2n))\times\BC^{2n}$ & $\circlearrowright\BC^{2n+1}_+\otimes\BC^{2n}_-$ & &
    $\sharp\fM^\iota_{2n+1,2n}$\\
    \hline
$\Sp(2n)\times\Sp(2m)\circlearrowright$  &
    $\SO(2m+1)\times\Sp(2n)$  &
    $\osp(2m+1|2n)$ & $\sharp\fM^\iota_{2m,2n}=$\\
    $\Sp(2n)\times\CS_{(2n-2m,1^{2m})}$ & $\circlearrowright\BC^{2m+1}_+\otimes\BC^{2n}_-$ & $m<n$ &
    $\sharp\fM^\iota_{2m+1,2n}$ \\
    \hline
$\Sp(2n)\times\Sp(2m)\circlearrowright$  &
    $\SO(2n+1)\times\Sp(2m)$  &
    $\osp(2n+1|2m)$ & $\sharp\fM^\iota_{2m,2n}=$\\
    $\Sp(2n)\times\CS_{(2n-2m,1^{2m})}$ & $\circlearrowright\BC^{2n+1}_+\otimes\BC^{2m}_-$ & $m<n$ &
    $\sharp\fM^\iota_{2m,2n+1}$ \\
    \hline
$\SO(2n+1)\times\SO(2m)\circlearrowright$  &
    $\SO(2m)\times\Sp(2n)$  &
    $\osp(2m|2n)$ & $\sharp\fM^\iota_{2m,2n}=$\\
    $\SO(2n+1)\times\CS_{(2n+1-2m,1^{2m})}$ & $\circlearrowright\BC^{2m}_+\otimes\BC^{2n}_-$ & $m\leq n$ &
    $\sharp\fM^\iota_{2m,2n+1}$ \\
    \hline
$\SO(2n)\times\SO(2m+1)\circlearrowright$  &
    $\SO(2n)\times\Sp(2m)$  &
    $\osp(2n|2m)$ & $\sharp\fM^\iota_{2m,2n}=$\\
    $\SO(2n)\times\CS_{(2n-1-2m,1^{2m+1})}$ & $\circlearrowright\BC^{2n}_+\otimes\BC^{2m}_-$ & $m<n$ &
    $\sharp\fM^\iota_{2m+1,2n}$ \\
    \hline
$\PSp(6)\times\PGL(2)\circlearrowright$  &
    $\Spin(7)\times\Sp(2)$  &
    ${\mathfrak f}(4)$ & $9$\\
    $\PSp(6)\times\CS_{(3,3)}$ & $\circlearrowright\BC^8_+\otimes\BC^2_-$ & & \\
    \hline
$\on{G}_2\times\SL(2)\circlearrowright$  &
    $\on{G}_2\times\Sp(2)$  &
    ${\mathfrak g}(3)$ & $7$\\
    $\on{G}_2\times\CS_{\on{short}}$ & $\circlearrowright\BC^7_+\otimes\BC^2_-$ & & \\
    \hline  
\end{tabular}
  \caption{Numbers of irreducible Lagrangian components}
  \label{tab}
\end{table}

It would be interesting to describe the resulting representations of the Weyl groups $W_\fg$,
and partitions of the sets of (centrally symmetric) mazes into microlocal and Kazhdan-Lusztig
cells, cf.~\cite[\S5]{t} for the mirabolic case
$G\circlearrowright\scrX=\GL(N)\times\GL(N)\circlearrowright T^*(\GL(N)\times\BC^N)$.

\subsection{Acknowledgments}
We are grateful to A.~Braverman, A.~Elashvili, A.~Hanany, M.~Jibladze, H.~Nakajima, V.~Serganova and
R.~Yang for very useful discussions.
The research of M.F.~was supported by the Israel Science Foundation (grant No.~994/24).

\section{Generalities}

\subsection{Coisotropic varieties and moment maps}

Let $B\subseteq G$ be a Borel subgroup with Lie algebra $\fb\subset\fg$
  and let $\CB=G/B$ be the flag variety.
  Associated with $\lambda\in \fg^*$ such that $\bar\lambda:=\lambda|_\fb$ is
  a character of $\fb$, there is a  twisted cotangent
  bundle $T^*_{\bar\lambda}\CB=G\times_B(\lambda+\fb^\perp)\to \CB$,
  where  $\fb^\perp\subset\fg^*$ denotes the annihilator of $\fb$ in $\fg^*$.

  Given a smooth  hamiltonian $G$-variety $\scrX$,
  we let $G$ act on $T^*_{\bar\lambda}\CB\times \scrX$ diagonally and write
  $\bmu\colon\scrX\to\fg^*$, resp.\ $\bmu_{T^*_{\bar\lambda}\CB\times \scrX}\colon
  T^*_{\bar\lambda}\CB\times \scrX\to \fg^*$, for the corresponding moment map. 

\begin{prop}
  \label{prop equiv}
  Let $\scrX$ be a smooth affine Hamiltonian $G$-variety,
  and $Z\subset\bmu(\scrX)$ a locally closed $G$-stable subset.
  Let $\bar\lambda\in\fb^*$ be a character of $\fb$, so that
  its preimage in $\fg^*$ is $\bar\lambda+\fb^\perp$.
  We assume that the intersection $Z_{\bar\lambda}:=Z\cap(\bar\lambda+\fb^\perp)$ is nonempty.
  
    Then, we have

    \noindent
    \textup{(1)}  The following are  equivalent:
 
\begin{enumerate}[(a)]
    \item\label{(1)}
      $\bmu^{-1}(Z_{\bar\lambda})$ is Lagrangian in $\scrX$;

    \item\label{(2)}
      For any $\lambda\in Z_{\bar\lambda}$, the fiber $\bmu^{-1}(\lambda)$ is isotropic in $\scrX$;

    \item\label{(3)}
      For any $\lambda\in Z_{\bar\lambda}$, all irreducible components of
    $\bmu^{-1}(\lambda)$ have dimension less than or equal to
    $\frac{1}{2}(\dim\scrX-\dim(G.\lambda))$, where
    $G.\lambda\subset \fg^*$ is the coadjoint orbit of $\lambda$;

 \item\label{(4)}
   Every irreducible component    of $\bmu_{T^*_{\bar\lambda}\CB\times \scrX}^{-1}(0)\times_Z\scrX$
      is Lagrangian in ${T^*_{\bar\lambda}\CB\times \scrX}$.
    \end{enumerate}

    \noindent
    \textup{(2)}     The variety $\scrX$ is  coisotropic if and only if 
    the equivalent conditions in \textup{(1)} hold for an open $G$-stable subset
    $Z\subset\bmu(\scrX)$ and every $\lambda$ such that the intersection
    $Z_{\bar\lambda}$ is nomempty.
\end{prop}

\begin{proof}
We denote by $\omega$ the symplectic form on $\scrX$.
For $x\in \scrX$ we write $E^\perp$ for the annihilator  of a subspace
$E\subset T_x\scrX$, resp. 
$\fg.x$ for the tangent space at $x\in\scrX$ to the $G$-orbit of $x$,
and $\fg_x$ for the Lie algebra of the stabilizer $G_x$ of $x$ in $G$.
It is well known that one has
\begin{equation}\label{fiber}
\Ker(d_x\mu)=(\fg.x)^\perp.
\end{equation}
Observe that the set $\lambda+\fb^\perp$
meets only finitely many coadjoint orbits.
Hence,~\ref{(1)} holds iff  for any such coadjoint orbit $O\subset \fg^*$ 
every irreducible component of $\bmu^{-1}(O\cap (\lambda+\fb^\perp))$
is Lagrangian.
Let $\Lambda$ be such a component and  $Y=\bmu(\Lambda)$. We denote by
$\bmu_\Lambda\colon \Lambda\to Y$ the restriction of $\bmu$ to $\Lambda$.
Without loss of generality we may assume that $\lambda$
is a sufficiently general point of $Y$.
The map $\bmu\colon \bmu^{-1}(O)\to O$ being a locally trivial fibration,
the set $Y$ must be a 
dense subset of an irreducible component of $O\cap (\lambda+\fb^\perp)$
and the set $F:=\Lambda\cap\bmu^{-1}(\lambda)$ must be dense in an irreducible component
  of $\bmu_\Lambda^{-1}(\lambda)$.
Choosing a general $x\in F$ we may assume that
the morphism $\bmu_\Lambda\colon \Lambda\to Y$ is smooth at $x$.
     Consider a vector subspace of $\fg$ defined by
      $V=\{v\in \fg \mid \mathrm{ad}_v(\lambda)\in T_\lambda (O\cap (\lambda+\fb^\perp))\}$.
      We have $T_\lambda Y=T_\lambda (O\cap (\lambda+\fb^\perp))=\mathrm{ad}_V(\lambda)$.
      We know that $O\cap (\lambda+\fb^\perp)$ is Lagrangian in $O$,
      by~\cite[Theorem 3.3.7]{cg}.
      Hence,   for any $v,v'\in V$ we have $\lambda([v,v'])=0$.

     Assume that~\ref{(2)} holds, so $F$ is isotropic in $X$.
     Thus, for any $\xi,\xi'\in T_xF$ we have $\omega(\xi,\xi')=0$.
   Since $\bmu_\Lambda$ is smooth at $x$ as a map to the image of $\Lambda$,
one has
$T_x\Lambda=T_xF+ V.x$.
For any  $v,v'\in V$, by properties of the moment maps we deduce
     $\omega(v.x,v'.x)=\lambda([v,v'])=0$.   
Further, since $T_xF\subset \Ker(d_x\bmu)$ and $v.x,v'.x\in \fg.x$, it follows   from~\eqref{fiber}
     that $\omega(\xi,v'.x)=\omega(v.x,\xi')=0$.
     Thus, we find
   $\omega(\xi+v.x,\, \xi'+ v'.x)=  0+0+0+0$,                       
proving  the implication~\ref{(2)} $\Rightarrow$~\ref{(1)}.
The opposite implication   is clear.

Observe next  that $\Lambda$ is  an   irreducible component of the fiber
  over $\lambda|_\fb\in\fb^*$ of the moment map for the $B$-action on $\scrX$.
 It follows that  $\Lambda$  is a coisotropic subvariety of
 $\scrX$, by~\cite[Theorem 1.5.7]{cg}.
 We deduce that $\Lambda$ is Lagrangian iff
 $\dim \Lambda=\frac{1}{2}\dim \scrX$.  On the other hand, 
since $O\cap (\lambda+\fb^\perp)$ is Lagrangian in $O$,
                                      we find  $\dim \Lambda=\dim F+\dim Y=\dim F+\frac{1}{2}\dim O$.
 The equivalence~\ref{(1)} $\Leftrightarrow$~\ref{(3)} follows.

   To prove the  equivalence~\ref{(1)} $\Leftrightarrow$~\ref{(4)}
we use the canonical isomorphism
   $T^*_{\bar\lambda}{\mathcal B}\times\scrX$ $\cong G\times_B\,\big((\lambda+\fb^\perp) \times \scrX\big)$.
   The subvariety $\bmu_{T^*_{\bar\lambda}{\mathcal B}\times\scrX}^{-1}(0)$
   goes under the isomorphism
   to  $G\times_B\,S$, where
   $S\subset (\lambda+\fb^\perp) \times\bmu^{-1}(\lambda+\fb^\perp)$
   is the graph of the map $\bmu\colon \bmu^{-1}(\lambda+\fb^\perp)\to\lambda+\fb^\perp$.
   Using this, one checks that 
   $\bmu_{T^*_{\bar\lambda}{\mathcal B}\times\scrX}^{-1}(0)$
   is Lagrangian  in $T^*_{\bar\lambda}{\mathcal B}\times_Z\scrX$ iff
   $\bmu^{-1}(\lambda+\fb^\perp)$ is Lagrangian  in $\scrX$. 

   It remains to prove (2). Let $x\in \scrX$ be a sufficiently general point
   so that the morphism $\bmu$ is smooth at $x$.
   It follows that the vector space
   $\Ker(d_x\bmu)$ equals  the tangent space at $x$ to the fiber of $\bmu$ that contains $x$.
   Since $ \Ker(d_x\bmu)=(\fg.x)^\perp$  we see  that general fibers of $\bmu$ are isotropic
   iff   for general $x\in \scrX$, the vector space $(\fg.x)^\perp$ is isotropic,
   equivalently,  the vector
   space  $\fg.x$ is coisotropic.
   The latter holds iff
   general $G$-orbits are coisotropic.
   We conclude that $\scrX$ is coisotropic 
      iff condition~\ref{(2)} of part (1) holds for general fibers.
\end{proof}

\begin{rem}
  Assume that a coisotropic $\scrX$ is equipped with an action of $\BG_m$ such that $\omega$
  has a positive weight, and the induced $\BG_m$-action contracts $\scrX/\!\!/G$
  to a point. 
  If $\bmu^{-1}(\fb^\perp)$ is Lagrangian (as conjectured in~\ref{weak}(a)),
  then~Proposition~\ref{prop equiv}(1) holds with $Z=\bmu(\scrX)$.
\end{rem}

\section{F(4)}

\subsection{Setup}
We consider a nilpotent element $e$ of Jordan type $(3,3)$ in the symplectic Lie algebra
$\fsp(6)$. Let $\CS_e$ be a Slodowy slice to the nilpotent orbit $\BO_e\subset\fsp(6)$.
A maximal reductive subgroup of the centralizer of $e$ in $\PSp(6)$ is $\PGL(2)$.
By the construction of~\cite[\S2]{l}, cf.~\cite[\S3.4]{bzsv}, the product
$\scrX:=\PSp(6)\times\CS_e$ carries a symplectic structure and a hamiltonian action of
$G:=\PSp(6)\times\PGL(2)$. This action is coisotropic~\cite[\S2.3]{fu}. It satisfies
the conditions~\cite[\S3.5.1]{bzsv} of hypersphericity.
Moreover, $\scrX$ is polarizable, i.e.\ $\scrX$ is a twisted cotangent bundle of the following
spherical $G$-variety $\scrY$.

We view $\PSp(6)$ as the quotient by $\{\pm1\}$ of the automorphism group $\Sp(6)$ of
$V=\BC^6$ equipped with a symplectic pairing
$\langle\ ,\ \rangle$. We choose a basis $v_1,\ldots,v_6$ in $V$ such that
$\langle v_i,v_{7-i}\rangle=1=-\langle v_{7-i},v_i\rangle$, $i=1,2,3$, and all the other products vanish.
Let $\on{P}\subset\PSp(6)$ be the stabilizer of the subspace $V_{12}=\BC v_1\oplus\BC v_2$.
Let $\on{U}$ be the unipotent radical of $\on{P}$. Let $\psi\colon\on{U}\to\BG_a$ be the
character corresponding to the character of its Lie algebra $\fu$ given by the sum of
matrix elements $u_{13}+u_{24}$. Then $\scrY=\PSp(6)/\on{U}$, and $\scrX=T^*_\psi\scrY$.

Note that $\scrY$ carries the (left) action of $\PSp(6)$ and the (right) commuting action
of the Levi quotient $\on{L}=\on{P}/\on{U}$ isomorphic to the quotient of $\GL(2)\times\SL(2)$
modulo the diagonal copy of $\{\pm1\}$. The Levi quotient $\on{L}$ contains the diagonal
copy of $\SL(2)/\{\pm1\}=\PGL(2)$ that stabilizes the character $\psi$. In other words,
$\psi$ extends to the same named character of $H:=\PGL(2)\ltimes\on{U}$. We will view
$H$ as a subgroup of $G$ (with respect to the diagonal embedding
$H\hookrightarrow\PSp(6)\times\PGL(2)$). Then $\scrY=G/H$.

The recipe of~\cite[\S4]{bzsv} produces from $G\circlearrowright\scrX$ the dual symplectic
variety $\scrX^\vee$ with a hamiltonian coisotropic action of $G^\vee=\Spin(7)\times\SL(2)$.
Namely, $\scrX^\vee$ is a symplectic vector representation of $G^\vee$ equal to the tensor product
of the spinor representation $\BC^8_+\circlearrowleft\Spin(7)$ (with its invariant symmetric
bilinear form) and the tautological 2-dimensional representation $\BC^2_-\circlearrowleft\SL(2)$
(with its invariant skew-symmetric bilinear form).

Let $\fb\subset\fg=\on{Lie}G,\ \fb^\svee\subset\fg^\svee=\on{Lie}G^\vee$ be Borel subalgebras.
Let $\fb^\perp\subset\fg^*,\ (\fb^\svee)^\perp\subset(\fg^\svee)^*$ be their annihilators.
Let $\bmu\colon\scrX\to\fg^*,\ \bmu^\svee\colon\scrX^\vee\to(\fg^\svee)^*$ be the moment maps.
Let $\Lambda_\scrX=\bmu^{-1}(\fb^\perp),\ \Lambda_{\scrX^\vee}=(\bmu^\svee)^{-1}((\fb^\svee)^\perp)$.

\begin{prop}
  \label{f4 lagr}
  Both $\Lambda_\scrX$ and $\Lambda_{\scrX^\vee}$ are Lagrangian subvarieties with 9 irreducible
  components.
\end{prop}

The proof occupies~\S\S\ref{f4 slice},\ref{f4 linear}.

\subsection{Equivariant slice}
\label{f4 slice}
First we count the irreducible components of $\Lambda_\scrX$.

The irreducible components of $\Lambda_\scrX$ are the twisted conormal bundles to the
{\em relevant} $B$-orbits in $\scrY$ (where $B\subset G$ is the Borel subgroup with Lie algebra
$\fb$). Equivalently, we consider the flag variety $\CB=G/B$ with the action of $H$.
Then an $H$-orbit in $\CB$ is called relevant if the restriction of $\psi$ to the
stabilizer of a point in this orbit is trivial. Let us give yet one more equivalent
but more economical reformulation. Let $B_{\PGL(2)}\subset\PGL(2),\ B_{\PSp(6)}\subset\PSp(6)$
be the upper-triangular Borel subgroups. Let
$H_{\PSp(6)}:=B_{\PGL(2)}\ltimes\on{U}\subset\PSp(6)$ (we view $B_{\PGL(2)}\subset\PGL(2)$
as a (diagonal) subgroup of $\on{L}=(\GL(2)\times\SL(2))/\{\pm1\}$).
Then $H_{\PSp(6)}$ acts on the flag variety
$\CB_{\PSp(6)}=\PSp(6)/B_{\PSp(6)}$ with finitely many orbits, and we have to count the relevant
$H_{\PSp(6)}$-orbits.

Since $H_{\PSp(6)}\subset B_{\PSp(6)}$, the $H_{\PSp(6)}$-orbits on $\CB_{\PSp(6)}$ are contained in
the $B_{\PSp(6)}$-orbits (Schubert cells), indexed by the Weyl group $W$ of $\PSp(6)$.
Let $U_{\PSp(6)}$ be the unipotent radical of $B_{\PSp(6)}$, and let $U'$ be the unipotent
radical of $H_{\PSp(6)}$. We have the embeddings of unipotent groups
$\on{U}\subset U'\subset U_{\PSp(6)}$ of dimensions~7,8,9 respectively.

The set $R_+$ of positive roots (of $\fu_{\PSp(6)}$) consists of
$2\epsilon_1,2\epsilon_2,2\epsilon_3,
\epsilon_1-\epsilon_2,\epsilon_1-\epsilon_3,\epsilon_2-\epsilon_3,
\epsilon_1+\epsilon_2,\epsilon_1+\epsilon_3,\epsilon_2+\epsilon_3$. Among them the simple
ones are $\alpha_1=\epsilon_1-\epsilon_2,\ \alpha_2=\epsilon_2-\epsilon_3,\ \alpha_3=2\epsilon_3$.
The Lie subalgebra $\on{Lie}\on{U}=\fu\subset\fu_{\PSp(6)}$ is spanned by the root subspaces
of $\fu_{\PSp(6)}$ with the exception of $\alpha_1$ and $\alpha_3$.
The character $\psi$ is nontrivial on the root subspaces
$\epsilon_1-\epsilon_3,\epsilon_2+\epsilon_3$ of $\fu$.

The Weyl group $W$ is generated by the simple reflections $s_1,s_2,s_3$ and has~48 elements.
For $w\in W$ we denote by $X_w\subset\CB_{\PSp(6)}$ the corresponding Schubert cell of dimension
$\ell(w)$. We denote by $x_w\in X_w$ the unique point fixed by the diagonal (in the basis
$v_1,\ldots,v_6$) torus $T\subset\PSp(6)$. We denote by $Y_w\subset X_w$ the orbit $U'\cdot x_w$.
It coincides with the orbit $H_{\PSp(6)}\cdot x_w$ since $x_w=T\cdot x_w$, and
$H_{\PSp(6)}=T'\ltimes U'$ for the one-parametric subgroup $T'\subset T$ corresponding to the
coweight $\epsilon_1^*-\epsilon_2^*+\epsilon_3^*$. If $Y_w\subsetneqq X_w$ is a proper inclusion,
then the dimension $\dim Y_w$ must be equal to $\dim X_w-1$, and the complement
$Y'_w=X_w\setminus Y_w$ must form a single $H_{\PSp(6)}$-orbit.

This orbit splitting $X_w=Y_w\sqcup Y'_w$ occurs if and only if $w$ is of maximal length
in the coset $wW_{13}$ of the parabolic subgroup $W_{13}\subset W$ generated by $s_1,s_3$.
Indeed, the set of orbits $B_{\PSp(6)}\backslash\CB_{\PSp(6)}$ is in natural bijection with the set
of $\PSp(6)$-orbits in the set of pairs $(V_\bullet^{(1)},V_\bullet^{(2)})$, where $V_\bullet^{(1,2)}$
are complete self-orthogonal flags in $V$. The set of orbits $H_{\PSp(6)}\backslash\CB_{\PSp(6)}$
is in natural bijection with the set of $\PSp(6)$-orbits in the set of triples
$(V_\bullet^{(1)},V_\bullet^{(2)},\varphi)$, where $\varphi$ is an isomorphism
$V_4^{(1)}/V_2^{(1)}\iso V_2^{(1)}$ taking $V_3^{(1)}/V_2^{(1)}$ to $V_1^{(1)}$.
Now $X_w$ splits into two orbits $X_w=Y_w\sqcup Y'_w$ iff in both subquotients $V_4^{(1)}/V_2^{(1)}$
and $V_2^{(1)}$, the second flag $V_\bullet^{(2)}$ cuts out the lines distinct from
$V_3^{(1)}/V_2^{(1)}$ and $V_1^{(1)}$ respectively (that is, along with the latter lines, giving
rise to projective bases in these subquotients): the splitting is
according to the dichotomy if $\varphi$ takes the former base to the latter one or if it does not.
Finally, the second flag $V_\bullet^{(2)}$ gives
rise to projective bases in both subquotients $V_4^{(1)}/V_2^{(1)}$ and $V_2^{(1)}$ iff
$w$ is of maximal length in the coset $wW_{13}$.

The stabilizer of $x_w$ in $U'$ is equal to $U'\cap U_{\PSp(6)}^w$. The character $\psi$ vanishes
on this intersection iff $\epsilon_1-\epsilon_3\not\in wR_+\not\ni\epsilon_2+\epsilon_3$.
In other words, the orbit $Y_w$ is relevant iff
$w^{-1}(\epsilon_1-\epsilon_3)\not\in R_+\not\ni w^{-1}(\epsilon_2+\epsilon_3)$.
By inspection, this only happens for $w_1=s_2s_1s_3s_2s_1=:s_{21321},\ w_2=s_{213213},\
w_3=s_{1232132},\ w_4=s_{2321232},\ w_5=s_{32123213},\ w_6=s_{23212323},\ w_7=s_{23212321},\ w_8=w_0$ (the longest
element). Of these~8 elements, $w_i$ is maximal in $w_iW_{13}$ iff $i=2,5,8$. Hence for these~3
values of $i$, the orbit $X_{w_i}$ splits into $Y_{w_i}\sqcup Y'_{w_i}$.

Note that if $X_w=Y_w\sqcup Y'_w$ and $Y_w$ is irrelevant, then $Y'_w$ is irrelevant as well.
Indeed, the stabilizer of $x_w$ in $U'$ is the limit of the stabilizers $\on{Stab}_{U'}(x'_w)$ of
$x'_w\in Y'_w$ as $x'_w$ tends to $x_w$ (in particular, all these stabilizers have the same dimension).
Hence if $\psi$ is nontrivial on $\on{Stab}_{U'}(x_w)$, then it must be nontrivial on
$\on{Stab}_{U'}(x'_w)$ as well.

We conclude that any relevant orbit must be one of the following list:
$Y_{w_i},\ i=1,\ldots,8,\ Y'_{w_i},\ i=2,5,8$. It remains to study the relevance of
$Y'_{w_2},Y'_{w_5},Y'_{w_8}$. Clearly, the open orbit $Y'_{w_8}$ is relevant since the stabilizer of
any point there is trivial. For $Y'_{w_2}$, we choose $x'_{w_2}\in Y'_{w_2}$ as $u\cdot x_w$ for
a general $u\in U_{\PSp(6)}$. The Lie algebra of $\on{Stab}_{U'}(x_{w_2})=\on{Stab}_{U_{\PSp(6)}}(x_{w_2})$
is spanned by the root spaces $2\epsilon_3,\epsilon_1-\epsilon_2,\epsilon_1+\epsilon_3$.
After conjugation with a general $u\in U_{\PSp(6)}$ the root space $\epsilon_1-\epsilon_2$ will have
nontrivial projection onto the root space $\epsilon_1-\epsilon_3$, and hence $\psi$ will be nontrivial.
Thus $Y'_{w_2}$ is irrelevant.

Finally, for $Y'_{w_5}$ we also choose $x'_{w_5}\in Y'_{w_5}$ as $u\cdot x_w$ for
a general $u\in U_{\PSp(6)}$. The Lie algebra of $\on{Stab}_{U'}(x_{w_5})=\on{Stab}_{U_{\PSp(6)}}(x_{w_5})$
is spanned by the root space $\epsilon_2-\epsilon_3$. After conjugation with a general
$u\in U_{\PSp(6)}$ the root space $\epsilon_2-\epsilon_3$ will have
nontrivial projection onto the root space $\epsilon_2+\epsilon_3$, and hence $\psi$ will be nontrivial.
Thus $Y'_{w_5}$ is irrelevant.

All in all, we have~9 relevant orbits $Y'_{w_8},Y_{w_i},\ i=1,\ldots,8$.
We conclude that $\Lambda_\scrX$ has 9 irreducible components.

\subsection{Linear side}
\label{f4 linear}
We denote by $\CN\subset\scrX^\vee$ the moment map preimage $(\bmu^\svee)^{-1}(\CN_{\fg^\svee})$
of the nilpotent cone in $\fg^\vee\simeq(\fg^\vee)^*$.
We learned the statement of the following lemma from A.~Elashvili and M.~Jibladze.

\begin{lem}
  $G^\vee$ has~6 orbits in $\CN\colon \BO_{15},\BO_{13},\BO_{11},\BO_9,\BO_8,\BO_0$ such that
  $\dim\BO_i=i$.
\end{lem}

\begin{proof}
For a point $h\in\scrX^\vee=\BC^8_+\otimes\BC^2_-\cong\Hom(\BC^2_-,\BC^8_+)$, the condition
$h\in\CN$ is equivalent to $h^*h\in\CN_{\fsl(2)}$. This in turn is equivalent to the condition
that the rank of the symmetric form $(\ ,\ )$ on $\BC^8_+$ restricted to $\on{Im}(h)$ is at most~1.

The orbit $\BO_0$ consists of the point $h=0$.
The orbit $\BO_8$ is formed by all the homomorphisms $h$ with 1-dimensional isotropic image.
The orbit $\BO_9$ is formed by all the homomorphisms $h$ with 1-dimensional non-isotropic image.
Indeed, the quadric $Q^6$ of all isotropic lines in $\BC^8_+$ forms a single $\Spin(7)$-orbit
isomorphic to the isotropic Grassmannian $\on{IGr}(3,7)$. The complement $\BP^7\setminus Q^6$
(formed by all non-isotropic lines) also forms a single $\Spin(7)$-orbit with stabilizer of
a point isomorphic to $\on{G}_2\subset\Spin(7)$.

We consider the set of all injective $h$ with isotropic image. It is a 4-dimensional vector
bundle over the 9-dimensional isotropic Grassmannian $\on{IGr}(2,8)$. The action of $\Spin(7)$ on
$\on{IGr}(2,8)$ has two orbits: the closed 7-dimensional one isomorphic to $\on{IGr}(2,7)$,
and the open complement. Indeed, the group $\on{Aut}(\Spin(8))$ of all automorphisms of $\Spin(8)$
acts on $\on{IGr}(2,8)$. An outer automorphism of triality group $\fS_3\subset\on{Aut}(\Spin(8))$
takes $\Spin(7)\subset\Spin(8)$ into another copy $\Spin(7)'\subset\Spin(8)$. The restriction
of the tautological 8-dimensional representation $\BC^8_+$ of $\Spin(8)$ to $\Spin(7)'$ is the
direct sum of its tautological 7-dimensional representation $\BC^7_+$ and the trivial 1-dimensional
one $\BC^1_+$.
Now the closed 7-dimensional orbit is formed by the isotropic planes contained in $\BC^7_+$,
and the complementary 9-dimensional orbit is formed by the isotropic planes transversal to
$\BC^7_+$. They give rise to the orbits $\BO_{11}$ and $\BO_{13}$.

It remains to consider the set of all injective $h$ such that $\on{rk}(\ ,\ )|_{\on{Im}(h)}=1$.
We choose a non-isotropic line $\ell\subset\on{Im}(h)$. Its stabilizer $\on{Stab}_{\Spin(7)}(\ell)$
is isomorphic to $\on{G}_2\subset\Spin(7)$. The action of $\on{G}_2$ on the quadric $Q^6$ of
isotropic lines in $\BC^8_+$ has two orbits: the closed one, isomorphic to a quadric $Q^5\subset Q^6$,
and the open one: $Q^6\setminus Q^5$. The kernel $K\subset\on{Im}(h)$ of $(\ ,\ )|_{\on{Im}(h)}$
cannot lie in $Q^6\setminus Q^5$: otherwise the form $(\ ,\ )$ will be {\em non}-degenerate
on $K\oplus\ell$. Hence all the choices of $K$ lie in a single $\on{G}_2$-orbit. It follows that
all the planes in $\BC^8_+$ such that the restriction of $(\ ,\ )$ to the plane has rank exactly~1,
form a single $\Spin(7)$-orbit of dimension~11. It gives rise to the open orbit $\BO_{15}\subset\CN$.
\end{proof}

Recall that we identify $(\fg^\svee)^*$ with $\fg^\svee=\fso(7)\oplus\fsl(2)$.
The nilpotent orbits in $\fso(7),\fsl(2)$ are indexed by their Jordan types.
Then $\bmu^\svee(\BO_0)=0=\BO_{(1^7)}\times\BO_{(1^2)},\ \bmu^\svee(\BO_8)=0=\BO_{(1^7)}\times\BO_{(1^2)},\ 
\bmu^\svee(\BO_9)=\BO_{(1^7)}\times\BO_{(2)},\ \bmu^\svee(\BO_{11})=\BO_{(2^2,1^3)}\times\BO_{(1^2)},\ 
\bmu^\svee(\BO_{13})=\BO_{(3,1^4)}\times\BO_{(1^2)},\ \bmu^\svee(\BO_{15})=\BO_{(3,2^2)}\times\BO_{(2)}.$

Thus $\BO_8$ forms a Lagrangian irreducible component of $(\bmu^\svee)^{-1}((\fb^\svee)^\perp)=
(\bmu^\svee)^{-1}(\fn^\svee)$ where $\fn^\svee$ stands for the nilpotent radical of $\fb^\svee$.
The fibers of $\bmu^\svee\colon\BO_9\to\BO_{(1^7)}\times\BO_{(2)}$ are connected 7-dimensional,
and the intersection $\BO_{(1^7)}\times\BO_{(2)}\cap\fn^\svee$ has a unique 1-dimensional component
(a punctured line in the nilpotent cone of $\fsl(2)$), so $\BO_9\cap(\bmu^\svee)^{-1}(\fn^\svee)$
forms a single Lagrangian component. The fibers of
$\bmu^\svee\colon\BO_{11}\to\BO_{(2^2,1^3)}\times\BO_{(1^2)}$ are 3-dimensional, and the intersection
$\BO_{(2^2,1^3)}\times\BO_{(1^2)}\cap\fn^\svee$ is 4-dimensional, so
$\BO_{11}\cap(\bmu^\svee)^{-1}(\fn^\svee)$ is 7-dimensional and does not contribute to the Lagrangian
components.

The fibers of $\bmu^\svee\colon\BO_{15}\to\BO_{(3,2^2)}\times\BO_{(2)}$ are connected 1-dimensional,
and the intersection $\BO_{(3,2^2)}\times\BO_{(2)}\cap\fn^\svee$ has three 7-dimensional components,
see e.g.~\cite[Table at page 235]{s}. Hence $\BO_{15}\cap(\bmu^\svee)^{-1}(\fn^\svee)$ contributes~3
Lagrangian components.

The fibers of $\bmu^\svee\colon\BO_{13}\to\BO_{(3,1^4)}\times\BO_{(2)}$ are 3-dimensional with~2
connected components. Indeed, a maximal reductive subgroup of
$\on{Stab}_{G^\vee}(y),\ y\in\BO_{(3,1^4)}\times\BO_{(2)}$, is $\on{Pin}(4)$. It acts simply transitively
on $\BO_{13}\cap(\bmu^\svee)^{-1}(y)$. The fundamental group $\pi_1(\BO_{(3,1^4)})\simeq\BZ/2\BZ$
permutes these two connected components. The intersection $\BO_{(3,1^4)}\times\BO_{(2)}\cap\fn^\svee$
has three 5-dimensional components, see e.g.~\cite[Table at page 235]{s}. In order to count
the Lagrangian components of $\BO_{13}\cap(\bmu^\svee)^{-1}(\fn^\svee)$, we reformulate our problem
as follows. Let $\CB^\vee$ be the flag variety of $G^\vee$ parametrizing all the possible choices
of $\fb^\svee\subset\fg^\svee$. Let \[\CM=\{(\fb^\svee\in\CB^\vee,\
x\in\BO_{13}\cap(\bmu^\svee)^{-1}(\BO_{(3,1^4)}\times\BO_{(2)}\cap\on{rad}\fb^\svee))\}.\]
We have the natural projections
\(\CB^\vee\xleftarrow{\varpi}\CM\xrightarrow{\pi}\BO_{(3,1^4)}\times\BO_{(2)}.\)
All the fibers of $\varpi$ are isomorphic to $\BO_{13}\cap(\bmu^\svee)^{-1}(\fn^\svee)$, and
$\CB^\vee$ is simply connected. Hence the number of irreducible components of 
$\BO_{13}\cap(\bmu^\svee)^{-1}(\fn^\svee)$ is equal to the number of irreducible components of $\CM$.
On the other hand, given $y\in\BO_{(3,1^4)}\times\BO_{(2)}$, the fiber $\pi^{-1}(y)$ is
isomorphic to the product $\on{Spr}_y\times(\BO_{13}\cap(\bmu^\svee)^{-1}(y))$, where
$\on{Spr}_y\subset\CB^\vee$ is the Springer fiber. It has~4 irreducible components
$C_1,C_2,C_3,C_4$, and $\pi_1(\BO_{(3,1^4)})=\BZ/2\BZ$ acts by permutation of $C_3$ and $C_4$,
see e.g.~\cite[Table at page 235]{s}. Recall that $\BO_{13}\cap(\bmu^\svee)^{-1}(y)$ has two
irreducible components $D_1,D_2$ that are permuted by $\pi_1(\BO_{(3,1^4)})$. Hence the fiber
$\pi^{-1}(y)$ has~8 irreducible components $C_i\times D_j$, and $\pi_1(\BO_{(3,1^4)})$ acts
diagonally on $\on{Irr}\pi^{-1}(y)$. We see that the action of $\pi_1(\BO_{(3,1^4)})$
on $\on{Irr}\pi^{-1}(y)$ has~4 orbits, and thus $\CM$ has~4 irreducible components as well.
We conclude that $\BO_{13}\cap(\bmu^\svee)^{-1}(\fn^\svee)$ contributes~4 Lagrangian components.

All in all, $(\bmu^\svee)^{-1}((\fb^\svee)^\perp)$ has~9 irreducible components. This completes the
proof of~Proposition~\ref{f4 lagr}. \hfill $\Box$

\section{$\GL(M|N)$}
\label{glmn}

\subsection{Setup}
Given $M<N$, we consider a nilpotent element $e$ of Jordan type $(N-M,1^M)$ in the Lie algebra
$\fgl(N)$. Let $\CS_e$ be a Slodowy slice to the nilpotent orbit $\BO_e\subset\fgl(N)$.
The centralizer of $e$ in $\GL(N)$ contains $\GL(M)$.
By the construction of~\cite[\S2]{l}, cf.~\cite[\S3.4]{bzsv}, the product
$\scrX_{M,N}:=\GL(N)\times\CS_e$ carries a symplectic structure and a hamiltonian action of
$G:=\GL(N)\times\GL(M)$. For $M=N$, we define $\scrX_{N,N}:=T^*\GL(N)\times T^*\BC^N$
(cotangent bundle of the tautological representation).
The action $G\circlearrowright\scrX_{M,N}$ is coisotropic, see e.g.~\cite[\S2.2]{fu}. It satisfies
the conditions~\cite[\S3.5.1]{bzsv} of hypersphericity.
Moreover, $\scrX_{M,N}$ is polarizable, i.e.\ $\scrX_{M,N}$ is a twisted cotangent bundle of the following
spherical $G$-variety $\scrY_{M,N}$.

For $M<N$, let $\on{P}_{M,N}\subset\GL(N)$ be the parabolic subgroup containing the upper triangular
Borel subgroup
and corresponding to the partition $(M+1,1^{N-M-1})$. It contains the upper left block $\GL(M)$.
The $\GL(M)$-invariant character $\psi$ of the unipotent radical $\on{U}_{M,N}$ of
$\on{P}_{M,N}$ corresponds to the same named character of its Lie algebra $\fu_{M,N}$ equal to
the sum of matrix elements $u_{M+1,M+2}+\ldots+u_{N-1,N}$. Hence $\psi$ extends to the character of
$H_{M,N}:=\GL(M)\ltimes\on{U}_{M,N}$. Then $\scrY_{M,N}=\GL(N)/\on{U}_{M,N}$, and $\scrX_{M,N}=T^*_\psi\scrY_{M,N}$.
For $M=N$, we set $\scrY_{M,N}=\GL(N)\times\BC^N$.

Note that $\scrY_{M,N}$ carries the (left) action of $\GL(N)$ and the (right) commuting action of $\GL(M)$.
We will view $H_{M,N}$ as a subgroup of $G$ (with respect to the diagonal embedding
$H_{M,N}\hookrightarrow\GL(N)\times\GL(M)$). Then $\scrY_{M,N}=G/H_{M,N}$ (for $M<N$).

The recipe of~\cite[\S4]{bzsv} produces from $G\circlearrowright\scrX_{M,N}$ the dual symplectic
variety $\scrX_{M,N}^\vee$ with a hamiltonian coisotropic action of $G^\vee\simeq G$.
Namely, $\scrX_{M,N}^\vee$ is a symplectic vector representation of $G^\vee$ equal to
$T^*\Hom(\BC^N,\BC^M)$.

Let $\fb\subset\fg=\on{Lie}G,\ \fb^\svee\subset\fg^\svee=\on{Lie}G^\vee$ be the upper-triangular
Borel subalgebras, and $B\subset G,\ B^\vee\subset G^\vee$ the corresponding Borel subgroups.
Let $\fb^\perp\subset\fg^*,\ (\fb^\svee)^\perp\subset(\fg^\svee)^*$ be their annihilators.
Let $\bmu\colon\scrX_{M,N}\to\fg^*,\ \bmu^\svee\colon\scrX_{M,N}^\vee\to(\fg^\svee)^*$ be the moment maps.
Let $\Lambda_{\scrX_{M,N}}=\bmu^{-1}(\fb^\perp),\ \Lambda_{\scrX_{M,N}^\vee}=(\bmu^\svee)^{-1}((\fb^\svee)^\perp)$.

Since $\scrX_{M,N}^\vee$ is the cotangent bundle to $\Hom(\BC^N,\BC^M)$, the zero level of the
moment map $\Lambda_{\scrX_{M,N}^\vee}$ is the union of conormal bundles to $B^\vee$-orbits in
$\Hom(\BC^N,\BC^M)$. These orbits are naturally indexed by the set $\fR_{M,N}$ of non-attacking
rooks placements in the $M\times N$-chessboard. Indeed, the matrices with 1's at the rooks'
positions and 0's everywhere else, form the set of representatives of $B^\vee$-orbits in
$\Hom(\BC^N,\BC^M)$ (row reduction). 

\begin{prop}
  \label{gl(M|N) lagr}
  Both $\Lambda_{\scrX_{M,N}}$ and $\Lambda_{\scrX_{M,N}^\vee}$ are Lagrangian subvarieties with irreducible
  components naturally indexed by $\fR_{M,N}$.
\end{prop}

The proof will be given in~\S\ref{slice GL} after some combinatorial preparation in~\S\ref{comb}.

\subsection{Rooks, colored permutations and mazes}
\label{comb}
We have $\scrX_{M,N}^\vee=T^*\Hom(\BC^N,\BC^M)\cong T^*\Hom(\BC^M,\BC^N)$, and the set of
$B^\vee$-orbits in $\Hom(\BC^M,\BC^N)$ is also naturally indexed by $\fR_{M,N}$.
Given a rooks placement $\frr\in\fR_{M,N}$ we denote by $\Lambda_\frr\subset\scrX_{M,N}^\vee$
(resp.\ by $\Lambda'_\frr\subset\scrX_{M,N}^\vee$) the irreducible component of $\Lambda_{\scrX_{M,N}^\vee}$
equal to the closure of the conormal bundle to the $B^\vee$-orbit in $\Hom(\BC^N,\BC^M)$
(resp.\ in $\Hom(\BC^M,\BC^N)$) indexed by $\frr$. We have $\Lambda'_\frr=\Lambda_{\frr'}$ for a certain
rooks placement $\frr'$. In order to describe the bijection $\frr\mapsto\frr'$ (see~\S\ref{fourier}
below), we introduce another set $\fM_{M,N}$ in a natural bijection with $\fR_{M,N}$.
This set will be also necessary for the study of orthosymplectic case in~\S\ref{osp}.

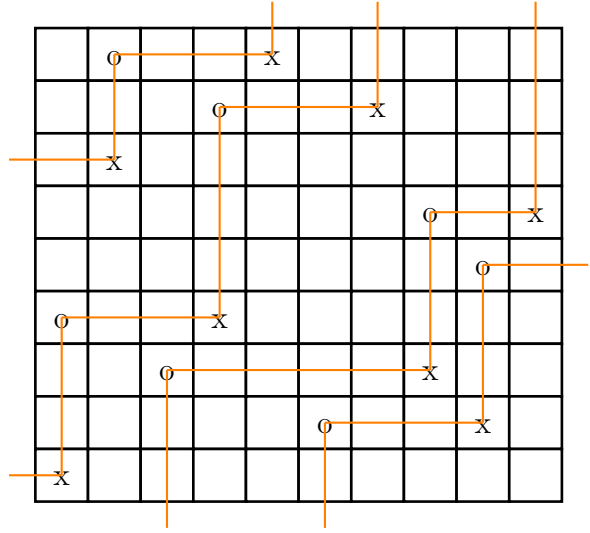
\begin{figure}[h]
\Yboxdim{1cm}
\begin{tikzpicture}[scale=.7]
  \Ystdtext1
  \Ylinethick{1pt}
\tyoung(0cm,0cm,\ o\ \ x\ \ \ \ \ ,\ \ \ o\ \ x\ \ \ ,\ x\ \ \ \ \ \ \ \ ,\ \ \ \ \ \ \ o\ x,\ \ \ \ \ \ \ \ o\ ,o\ \ x\ \ \ \ \ \ ,\ \ o\ \ \ \ x\ \ ,\ \ \ \ \ o\ \ x\ ,x\ \ \ \ \ \ \ \ \ )
\draw[orange,thick](4.5,.5)--(4.5,1.5);
\draw[orange,thick](1.5,.5)--(4.5,.5);
\draw[orange,thick](1.5,.5)--(1.5,-1.5);
\draw[orange,thick](1.5,-1.5)--(-0.5,-1.5);%
\draw[orange,thick](6.5,1.5)--(6.5,-0.5);
\draw[orange,thick](6.5,-0.5)--(3.5,-0.5);
\draw[orange,thick](3.5,-0.5)--(3.5,-4.5);
\draw[orange,thick](3.5,-4.5)--(0.5,-4.5);
\draw[orange,thick](0.5,-4.5)--(0.5,-7.5);
\draw[orange,thick](0.5,-7.5)--(-0.5,-7.5);%
\draw[orange,thick](9.5,1.5)--(9.5,-2.5);
\draw[orange,thick](9.5,-2.5)--(7.5,-2.5);
\draw[orange,thick](7.5,-2.5)--(7.5,-5.5);
\draw[orange,thick](7.5,-5.5)--(2.5,-5.5);
\draw[orange,thick](2.5,-5.5)--(2.5,-8.5);%
\draw[orange,thick](10.5,-3.5)--(8.5,-3.5);
\draw[orange,thick](8.5,-3.5)--(8.5,-6.5);
\draw[orange,thick](8.5,-6.5)--(5.5,-6.5);
\draw[orange,thick](5.5,-6.5)--(5.5,-8.5);
\end{tikzpicture}
\caption{A $9\times10$-maze}
\label{a maze}
\end{figure}

\begin{defn}
  \label{maze}
a) An $M\times N$-{\em maze} is a matrix (a chessboard) with $M$ rows and $N$ columns, equipped with a
collection of {\em walls}. Each wall is homeomorphic to an interval and consists of vertical or
horizontal
segments connecting the centers of boxes. Each wall ends (and starts) at the boundaries of the
chessboard. The different walls do not intersect. They satisfy the following requirements:

i) Each one of $N$ vertical lines through the centers of the boxes contains a vertical segment of
exactly one wall.  Each one of $M$ horizontal lines through the centers of the boxes contains a
horizontal segment of exactly one wall. 

ii) The right angle of any wall can occur only at the south-east (SE) corner (marked by x)
or at the north-west (NW) corner (marked by o). Hence the walls fall into the following four types:
type I ending at the left and upper boundaries of the chessboard; type II ending at the lower and
right boundaries of the chessboard; type III ending at the lower and upper boundaries of the
chessboard; type IV ending at the left and right boundaries of the chessboard. (Type III occurs
iff $N>M$, and type IV occurs iff $M>N$.)

The set of $M\times N$-mazes is denoted by $\fM_{M,N}$.

b) The involution $\iota\colon \fM_{M,N}\to \fM_{M,N}$ is obtained by the central symmetry
followed by replacing all x's by o's, and all o's by x's.
\end{defn}

We have a map $\pi_{M,N}\colon \fM_{M,N}\to \fR_{M,N}$ assigning to a maze the placement of rooks at boxes
marked with x.

\begin{lem}
  \label{pi MN}
  $\pi_{M,N}\colon \fM_{M,N}\to \fR_{M,N}$ is a bijection.
\end{lem}

\begin{proof}
We have to construct the inverse bijection $\varpi_{M,N}\colon \fR_{M,N}\to \fM_{M,N}$. To this end we add
the $M+1$-st row of length $N$ to the bottom of our chessboard, and also add the $N+1$-st column
of length $M$ to the right of our chessboard. Given a placement of rooks (an element of $\fR_{M,N}$),
we put x's into the boxes with rooks and also into all the
boxes of the last (added) row (resp.\ column) that are not in the vertical (resp.\ horizontal) lines
containing x's in the original $M\times N$ chessboard. The resulting set of x's is denoted $X_1$.
Now we consider the {\em NW envelope} $NW(X_1)$ of $X_1$: the union of north-west quadrants with
vertices in all the x's. The intersection of the boundary of $NW(X_1)$ with the original
$M\times N$-chessboard is the first wall $W_1$.

Then we erase all the x's in the south-east corners of $NW(X_1)$. The remaining set of
x's is denoted $X_2$. The intersection of the boundary of $NW(X_2)$ with the original
$M\times N$-chessboard is the second wall $W_2$. We erase all the x's in the south-east corners of
$NW_2$ to obtain $X_3$, and continue like this till there are no more x's. The resulting set of
walls $W_1,W_2,\ldots$ is the desired maze in $\fM_{M,N}$.
\end{proof}

Another bijection $\fM_{M,N}\iso \fR_{M,N}$ assigns to a maze the placement of rooks at boxes marked with o.
We will also need the following ``mixed'' versions of these bijections.
Assume our $M\times N$-chessboard is cut into two by a horizontal cut (so that $M=M_u+M_d$) or
by a vertical cut (so that $N=N_l+N_r$). The map $\pi_{M,N_l,N_r}\colon \fM_{M,N}\to \fR_{M,N}$ assigns to a
maze the placement of rooks at boxes marked with x in the left $M\times N_l$-chessboard and rooks
marked with o in the right $M\times N_r$-chessboard. The map $\pi_{M_u,M_d,N}\colon \fM_{M,N}\to \fR_{M,N}$
assigns to a maze the placement of rooks at boxes marked with x in the upper $M_u\times N$-chessboard
and rooks marked with o in the lower $M_d\times N$-chessboard. 

\begin{lem}
  \label{pi MNN}
  \textup{(a)} $\pi_{M,N_l,N_r}\colon \fM_{M,N}\to \fR_{M,N}$ is a bijection.

  \textup{(b)} $\pi_{M_u,M_d,N}\colon \fM_{M,N}\to \fR_{M,N}$ is a bijection.
\end{lem}

\begin{proof}
We prove (a), the proof of (b) being absolutely similar. 
We construct the inverse bijection $\varpi_{M,N_l,N_r}\colon \fR_{M,N}\to \fM_{M,N}$.
Given a placement of rooks (an element of $\fR_{M,N}$), we consider its left part (a placement of
rooks in the left $M\times N_l$-chessboard), and its right part (a placement of
rooks in the right $M\times N_r$-chessboard). We remove from the left chessboard (temporarily)
all the rows occupied by the rooks in the right chessboard. We also remove from the right chessboard
(temporarily) all the rows occupied by the rooks in the left chessboard. We denote the resulting
chessboards with the rooks placements $\frr_l$ and $\frr_r$. We consider the maze $\varpi(\frr_l)$ in the
left chessboard. In the right chessboard we consider the maze $\iota\varpi(\varsigma \frr_r)$ where
$\varsigma \frr_r$ stands for the centrally symmetric rooks placement, $\iota$ was defined
in~\ref{maze}b), and $\varpi$ was considered in the proof of~Lemma~\ref{pi MN}.

Now we restitute the rows deleted from the left and right chessboards, and extend (the vertical
parts of the walls of) our mazes through the restituted rows in an obvious way. The resulting left
and right mazes form the desired single $M\times N$-maze. Note that this maze intersects the vertical
cut (between the columns with numbers $N_l,N_{l+1}$) in all the rows that contain rooks
neither in the left chessboard, nor in the right one.
\end{proof}

We denote by $\fM^\iota_{M,N}$ the set of $M\times N$-mazes invariant with respect to the
involution $\iota$. Note that $\fM^\iota_{M,N}$ is empty if both $M$ and $N$ are odd.

\begin{lem}
  \label{symmetric mazes}
  \textup{(a)} $\sharp \fM^\iota_{2m,2n}=\sharp \fM^\iota_{2m,2n+1}$.

  \textup{(b)} $\sharp \fM^\iota_{2m,2n}=\sharp \fM^\iota_{2m+1,2n}$.
\end{lem}

\begin{proof}
We prove (a), the proof of (b) being absolutely similar. Namely, we define a bijection
$\fM^\iota_{2m,2n}\iso \fM^\iota_{2m,2n+1}$. If a maze in $\fM^\iota_{2m,2n}$ is a disjoint
union of two mazes in $\fM_{2m,n}$ (left half and the right half of the $2m\times2n$-chessboard),
then we just insert the column in the middle and add the vertical wall of type III in this column.

Otherwise, we apply the bijection $\pi_{2m,n,n}$ of~Lemma~\ref{pi MNN} to our maze in
$\fM^\iota_{2m,2n}$ and obtain a centrally
symmetric placement of rooks in the $2m\times2n$-chessboard. We insert the column in the middle,
and place a rook in this column in the highest row among $m+1,\ldots,2m$ unoccupied by the
pre-existing rooks. (Note that if all the rows are already occupied, then we are in the situation
of the previous paragraph.) Then we apply the bijection $\varpi_{2m,n+1,n}$ of~Lemma~\ref{pi MNN} to
the resulting rooks placement.
\end{proof}

\begin{figure}[h]
\Yboxdim{1cm}
\begin{tikzpicture}[scale=.7]
  \Ystdtext1
  \Ylinethick{1pt}
\tyoung(0cm,0cm,o\ x\ ,x\ \ \ ,\ \ \ o,\ o\ x)
\draw[orange,thick](2.5,.5)--(2.5,1.5);
\draw[orange,thick](.5,.5)--(2.5,.5);
\draw[orange,thick](.5,.5)--(.5,-0.5);
\draw[orange,thick](.5,-0.5)--(-0.5,-0.5);%
\draw[orange,thick](4.5,-1.5)--(3.5,-1.5);
\draw[orange,thick](3.5,-1.5)--(3.5,-2.5);
\draw[orange,thick](3.5,-2.5)--(1.5,-2.5);
\draw[orange,thick](1.5,-2.5)--(1.5,-3.5);
\tyoung(5cm,0cm,\ \ \ \ \ ,x\ \ \ \ ,\ \ \ \ o,\ \ x\ \ )
\draw[blue,thick](8,1)--(8,-3);
\draw[black,dotted](5,-1.5)--(8,-1.5);
\draw[black,dotted](8,-0.5)--(10,-0.5);
\draw[black,dotted](8,-2.5)--(10,-2.5);
\tyoung(11cm,0cm,\ \ ox\ ,x\ \ \ \ ,\ \ \ \ o,\ ox\ \ )
\draw[orange,thick](11.5,-0.5)--(11.5,1.5);
\draw[orange,thick](11.5,-0.5)--(10.5,-0.5);%
\draw[orange,thick](15.5,-3.5)--(15.5,-1.5);
\draw[orange,thick](15.5,-1.5)--(16.5,-1.5);%
\draw[orange,thick](14.5,1.5)--(14.5,.5);
\draw[orange,thick](14.5,.5)--(13.5,.5);
\draw[orange,thick](13.5,.5)--(13.5,-2.5);
\draw[orange,thick](13.5,-2.5)--(12.5,-2.5);
\draw[orange,thick](12.5,-2.5)--(12.5,-3.5);
\end{tikzpicture}
\caption{$\fM_{4,4}^\iota\iso\fM_{4,5}^\iota$}
\label{bijection}
\end{figure}

Recall the notion of colored permutations~\cite[\S3.1]{t}.
Namely, a colored permutation $(w,\beta)$ is a pair consisting of a permutation $w\in\fS_N$ and
a subset $\beta\subset\{1,\ldots,N\}$ such that if $j\in\beta$ (is blue), and
$i\in\{1,\ldots,N\}\setminus\beta$ (is red), then either $i>j$ or $w(i)>w(j)$.
The set of colored (into red and blue)
permutations of $\{1,\ldots,N\}$ is denoted $RB_N$. We construct a map $\phi\colon \fM_{N,N}\to RB_N$
as follows.
Given an $N\times N$-maze we define the following colored permutation. Its graph (a collection of
boxes in the $N\times N$-chessboard) is the union of all the south-east corners (x's) of all the walls
of type I and of all the north-east corners (o's) of all the walls of type II. (Note that there
are no walls of types III,IV since $M=N$.) The blue set $\beta$ coincides with the set of x's.
It is immediate that the resulting pair $(w,\beta)$ satisfies the above conditions for a colored
permutation.

\begin{lem}
  $\phi\colon \fM_{N,N}\to RB_N$ is a bijection.
\end{lem}

\begin{proof}
Here is the inverse bijection $\varphi\colon RB_N\to \fM_{N,N}$. Given a colored permutation $(w,\beta)$,
let $B_1$ (resp.\ $R_1$) be the set of blue (resp.\ red) boxes in the graph of $w$.
We take the NW envelope $NW(B_1)$ and its boundary $W_1$, then erase all the blue boxes in the
south-east corners of $NW(B_1)$ to obtain $B_2$. We continue like this till there are no more blue
boxes. We obtain the set of walls $W_1,W_2,\ldots$ of type I.

Similarly, we take the SE envelope $SE(R_1)$ and its boundary $W'_1$, then erase all the red boxes
in the north-west corners of $SE(R_1)$ to obtain $R_2$. We continue like this till there are no more
red boxes. We obtain the set of walls $W'_1,W'_2,\ldots$ of type II. The desired maze is the union
of walls $W_1,W'_1,W_2,W'_2,\ldots$
\end{proof}

The following corollary is immediate from the construction of bijection $\phi$.

\begin{cor}
  \label{symp mirab}
  The bijection $\phi\colon \fM_{N,N}\to RB_N$ intertwines the involution
  $\iota\colon \fM_{N,N}\to \fM_{N,N}$ and the involution $\on{F}\colon RB_N\to RB_N$
  of~\cite[Proposition 11]{t}.
\end{cor}

For $M<N$, we denote by $RB^M_N\subset RB_N$ the set of all colored permutations $(w,\beta)\in RB_N$
such that $M+1$ is colored in blue, while $M+2,\ldots,N$ are all colored in red.
We denote by $RB_{M,N}\subset RB^M_N$ the set of all colored permutations $(w,\beta)\in RB^M_N$
such that $w(M+1)>w(M+2)>\ldots>w(N)$. We construct a map $\rho_{M,N}\colon RB_{M,N}\to \fR_{M,N}$ as
follows. First we consider the permutation $w'$ obtained from $w$ by rearranging the last
$N-M$ positions: $w'(K):=w(K),\ 1\leq K\leq M;\ w'(M+k):=w(N+1-k),\ 1\leq k\leq N-M$. Second,
we consider the coloring $\beta'$ obtained from $\beta$ by changing all the colors of $M+2,\ldots,N$
from red to blue. Third, we consider the rooks placement $\pi_{N,N}\varphi(w',\beta')\in \fR_{N,N}$.
Finally, we delete all the rooks in the rows $M+1,\ldots,N$, and obtain the desired rooks
placement in $\fR_{M,N}$.

\begin{lem}
  \label{rho}
  $\rho_{M,N}\colon RB_{M,N}\to \fR_{M,N}$ is a bijection.
\end{lem}

\begin{proof}
We construct the inverse bijection $\varrho_{M,N}\colon \fR_{M,N}\to RB_{M,N}$. Given a rooks placement
in the $M\times N$-chessboard, we consider this chessboard inside the $N\times N$-chessboard
(the upper $M$ rows). Applying the bijection $\varpi_{M,N}$ to our rooks placement, we obtain an
$M\times N$-maze. It has $N-M$ walls of type~III, and there is a unique way to extend this maze
to an $N\times N$-maze such that all these walls extend to the walls of type~I in the extended maze.
Clearly, in the rows $M+1,\ldots,N$, the resulting $N\times N$-maze has only x's. We apply the
bijection $\phi$ to obtain a colored permutation $(w',\beta')$ from this $N\times N$-maze.
Then we consider the permutation $w$ obtained from $w'$ by rearranging the last
$N-M$ positions: $w(K):=w'(K),\ 1\leq K\leq M;\ w(M+k):=w'(N+1-k),\ 1\leq k\leq N-M$.
Finally, we consider the coloring $\beta$ obtained from $\beta'$ by changing all the colors of
$M+2,\ldots,N$ from blue to red. The resulting colored permutation $(w,\beta)$ is the desired
element of $RB_{M,N}$.
\end{proof}

\subsection{Linear side}
\label{linear gl}
For the orthosymplectic applications we will need one more bijection between $\fM_{M,N}$ and
$\on{Irr}\Lambda_{\scrX^\vee_{M,N}}$ in case $MN$ is even. We will consider the case of even $N=2n$
in detail, the case of even $M=2m$ being entirely similar.

Recall that $B^\vee$ is upper-triangular in the standard base $v^-_1,\ldots,v^-_{2n}$ of $\BC^N$
and the standard base $v^+_1,\ldots,v^+_{M}$ of $\BC^M$. We denote by $\BC^n\subset\BC^N$ the
subspace spanned by $v^-_1,\ldots,v^-_n$. Then $\scrL:=\Hom(\BC^N/\BC^n,\BC^M)\oplus\Hom(\BC^M,\BC^n)$
is a Lagrangian subspace of the symplectic vector space
$\scrX_{M,N}^\vee=\Hom(\BC^N,\BC^M)\oplus\Hom(\BC^M,\BC^N)$ with the quotient
$\scrQ_{M,N}:=\scrX_{M,N}^\vee/\scrL=\Hom(\BC^n,\BC^M)\oplus\Hom(\BC^M,\BC^N/\BC^n)$.
Note that $\scrL$ is $B^\vee$-invariant but does not admit a $B^\vee$-invariant direct complement.
Choosing a direct complement $\scrL'\subset\scrX_{M,N}^\vee$ (e.g.\ arising from $\BC^{n\prime}\subset\BC^N$
spanned by $v^-_{n+1},\ldots,v^-_{N}$) we identify $\scrX_{M,N}^\vee$ with $T^*\scrQ_{M,N}$ (but in a {\em not}
$B^\vee$-equivariant way).

Given a point $\fq\in\scrQ_{M,N}$ and $b\in B^\vee$ such that $b\fq=\fq$, we choose a lift of $\fq$ to
$\fx\in\scrX_{M,N}^\vee$; then $b\fx-\fx\in\scrL$, and the symplectic pairing
$\langle\fx,b\fx-\fx\rangle=:\fq(b)$
is well defined (it depends on $\fq,b$ but not on a choice of $\fx$). We say that a $B^\vee$-orbit
$\BO\subset\scrQ_{M,N}$ is {\em relevant} if the above pairing vanishes for some (equivalently, any) point
$\fq\in\BO$ and any $b\in\on{Stab}_{B^\vee}(\fq)$.

Under the projection $\scrX_{M,N}^\vee\to\scrQ_{M,N}$, the image of $\Lambda_{\scrX_{M,N}^\vee}$ is the union of all
the relevant orbits. More precisely, let us consider the composition
$\on{q}\colon\Lambda_{\scrX_{M,N}^\vee}\hookrightarrow\scrX_{M,N}^\vee\twoheadrightarrow\scrQ_{M,N}$. Then
$\on{q}^{-1}(\BO)=\emptyset$ for a non-relevant $B^\vee$-orbit in $\scrQ_{M,N}$, and
$\on{q}^{-1}(\BO)$ is 
the conormal vector bundle $T^*_\BO\scrQ_{M,N}$
for a relevant $B^\vee$-orbit in $\scrQ_{M,N}$. Indeed, we identify
$\scrX^\vee_{M,N}\simeq T^*\scrQ_{M,N}\simeq\scrQ_{M,N}\times\scrL\simeq\scrL'\times\scrL$ and consider a point
$(y',y)\in\Lambda_{\scrX_{M,N}^\vee}$, so that $y'=\on{q}(y',y)$. We further decompose
$(y',y)=(y'_{Mn}+y'_{n\prime M},y_{Mn\prime}+y_{nM})$ for
\[y'_{Mn}\colon\BC^n\to\BC^M,\ y'_{n\prime M}\colon\BC^M\to\BC^{n\prime},\
y_{Mn\prime}\colon\BC^{n\prime}\to\BC^M,\ y_{nM}\colon\BC^M\to\BC^n.\]
Then the equation for $\Lambda_{\scrX_{M,N}^\vee}$ says that $y_{nM}y'_{Mn}\colon\BC^n\to\BC^n$ must be
strictly upper triangular, and $y'_{n\prime M}y_{Mn\prime}\colon\BC^{n\prime}\to\BC^{n\prime}$ must be
strictly upper triangular, and $y'_{Mn}y_{nM}+y_{Mn\prime}y'_{n\prime M}\colon\BC^M\to\BC^M$ must be
strictly upper triangular, while $y'_{n\prime M}y'_{Mn}\colon\BC^n\to\BC^{n\prime}$ must vanish
(since it is lower triangular in $\End(\BC^N)$).
There is no equation on $y_{Mn\prime}y_{nM}\colon\BC^{n\prime}\to\BC^n$ since it is automatically
upper triangular in $\End(\BC^N)$. Hence for a fixed $y'$, we only have linear commutation equation
on $y$. We have proved

\begin{lem}
  \label{relev}
  The Lagrangian subvariety $\Lambda_{\scrX_{M,N}^\vee}\subset\scrX_{M,N}^\vee\simeq T^*\scrQ_{M,N}$ is the union
  $\bigcup_{\BO\, \on{relevant}}T^*_\BO\scrQ_{M,N}$ of conormal bundles 
  over all the relevant orbits $\BO\subset\scrQ_{M,N}$. In particular, the set of relevant $B^\vee$-orbits in
  $\scrQ_{M,N}$ is in a natural bijection with $\on{Irr}\Lambda_{\scrX_{M,N}^\vee}$, and
  $\bigcup_{\BO\, \on{relevant}}T^*_\BO\scrQ_{M,N}$ is closed in $\scrX_{M,N}^\vee$. \hfill $\Box$
\end{lem}

\begin{rem}
  \label{super}
More generally, just within this remark, let us denote by $\scrX^\vee$ the odd part $\sg_{\bar1}$
of a basic classical Lie superalgebra $\sg$. Let us denote by $G^\vee$ a Lie group with the Lie
algebra $\sg_{\bar0}$ acting (via adjoint action) on $\scrX^\vee$. Then $\scrX^\vee$ is equipped with a
$G^\vee$-invariant symplectic form (restriction of the invariant non-degenerate even symmetric
bilinear form on $\sg$). The moment map $\bmu^\svee\colon\scrX^\vee\to\sg_{\bar0}^*\cong\sg_{\bar0}$
coincides with the half-supercommutator: $\bmu^\svee(x)=\frac12[x,x]$. We denote by $B^\vee\subset G^\vee$
a Borel subgroup, and $B^\vee_-$ an opposite Borel subgroup. Their Lie algebras are denoted
$\fb^\svee,\fb^\svee_-$, and their nilpotent radicals are denoted $\fn^\svee,\fn^\svee_-$.
We denote $(\bmu^\svee)^{-1}(\fn^\svee)$ by $\Lambda_{\scrX^\vee}$.
We choose a pair of opposite Borel subalgebras in $\sg$ containing $\fb^\svee,\fb^\svee_-$ and denote
their odd parts by $\sn,\sn_-$. Then $B^\vee$ acts on $\sn$ and on $\scrQ:=\scrX^\vee/\sn$. The direct
sum decomposition $\scrX^\vee=\sn\oplus\sn_-$ gives rise to an identification $\scrX^\vee\simeq T^*\scrQ$
(but this identification is {\em not} $B^\vee$-equivariant). We consider the composition
$\on{q}\colon\Lambda_{\scrX^\vee}\hookrightarrow\scrX^\vee\twoheadrightarrow\scrQ$. Then
$\on{q}^{-1}(\BO)=\emptyset$ for a non-relevant $B^\vee$-orbit in $\scrQ$, and $\on{q}^{-1}(\BO)$ is
a torsor $\CT^*_\BO\scrQ$ over the conormal bundle $T^*_\BO\scrQ$ for a relevant $B^\vee$-orbit in $\scrQ$.
Indeed, for $(x,x_-)\in\sn\oplus\sn_-\simeq T^*\scrQ$ we have $\frac12[x+x_-,x+x_-]=\frac12[x,x]+
\frac12[x_-,x_-]+[x,x_-]$. Since $\frac12[x,x]\in\fn^\svee$, the equation
$\frac12[x+x_-,x+x_-]\in\fn^\svee$ is equivalent to the affine-linear equation $[2x+x_-,x_-]=0$
on $x$ (for a fixed $x_-$).

We conclude that $\Lambda_{\scrX^\vee}=\bigcup_{\BO\, \on{relevant}}\CT^*_\BO\scrQ$ is Lagrangian, and
$\on{q}\colon\Lambda_{\scrX^\vee}\to\bigsqcup_{\BO\, \on{relevant}}\BO$ defines a bijection between
$\on{Irr}\Lambda_{\scrX^\vee}$ and the set of relevant $B^\vee$-orbits in $\scrQ$.
\end{rem}

Given a maze $\fm\in \fM_{M,N=2n}$ we take the matrix $C_\fm\in\Hom(\BC^n,\BC^M)$ with 1's at the
positions marked by x's in the left $M\times n$-half of $\fm$, and 0's everywhere else. We also take
the matrix $D_\fm\in\Hom(\BC^N/\BC^n,\BC^M)$ with 1's at the positions marked by o's in the right
$M\times n$-half of $\fm$, and 0's everywhere else. Then the transposed matrix
$^t\!D_\fm\in\Hom(\BC^M,\BC^N/\BC^n)$, and the pair $(C_\fm,{}^t\!D_\fm)\in\scrQ_{M,N}$.
Note that $^t\!D_\fm C_\fm=0=C_\fm{}^t\!D_\fm$.

In case $M=2m$, the above constructions have obvious analogues. For instance, we set
$\BC^N=\BC v_1^+\oplus\ldots\oplus\BC v_N^+,\ \BC^M=\BC v_1^-\oplus\ldots\oplus\BC v_{2m}^-,\ 
\BC^m=\BC v_{m+1}^-\oplus\ldots\oplus\BC v_{2m}^-\subset\BC^M,\
\scrL=\Hom(\BC^N,\BC^m)\oplus\Hom(\BC^M/\BC^m,\BC^N),\ \scrQ_{M,N}=\Hom(\BC^N,\BC^M/\BC^m)\oplus\Hom(\BC^m,\BC^N),\
D_\fm\in\Hom(\BC^N,\BC^M/\BC^m)$ has 1's at the positions marked by x's in the upper $m\times N$-half
of $\fm$, and 0's everywhere else. Finally, $C_\fm\in\Hom(\BC^N,\BC^m)$ has 1's at the positions
marked by o's in the lower $m\times N$-half of $\fm$, and 0's everywhere else. The transposed
matrix $^tC_\fm\in\Hom(\BC^m,\BC^N)$, and the pair $(D_\fm,{}^tC_\fm)\in\scrQ_{M,N}$.
Note that $D_\fm{}^tC_\fm=0={}^tC_\fm D_\fm$.

\begin{lem}
  \label{releva}
  \textup{(a)} For any maze $\fm\in \fM_{M,N=2n}$, the point $(C_\fm,{}^t\!D_\fm)\in\scrQ_{M,N}$ lies in a
  relevant orbit.

  \textup{(b)} If $\fm\ne\fm'\in \fM_{M,N=2n}$, then the points $(C_\fm,{}^t\!D_\fm)$ and
  $(C_{\fm'},{}^t\!D_{\fm'})$ lie in different relevant orbits. Thus, to a maze $\fm\in \fM_{M,N=2n}$
  we associate a relevant $B^\vee$-orbit $\BO_\fm\subset\scrQ_{M,N}$.

  \textup{(c)} Any relevant $B^\vee$-orbit in $\scrQ_{M,N=2n}$ is of type $\BO_\fm$ for an appropriate maze
  $\fm\in \fM_{M,N}$. Thus the association $\fm\mapsto\BO_\fm$ is a bijection between $\fM_{M,N}$
  and the set of relevant $B^\vee$-orbits in $\scrQ_{M,N}$.

  \textup{(d)} Similarly, the association $\fM_{M=2m,N}\ni\fm\mapsto$ the $B^\vee$-orbit
  $\BO_\fm$ of the point $(D_\fm,{}^tC_\fm)\in\scrQ_{M=2m,N}$, is a bijection between $\fM_{M,N}$
  and the set of relevant $B^\vee$-orbits in $\scrQ_{M,N}$.
\end{lem}

\begin{proof}
(a) We may view $C_\fm$ as an $M\times N$-matrix by adding to it the right $M\times n$-half of zeros,
and we may view $^t\!D_\fm$ as an $N\times M$-matrix by adding to it the upper $n\times M$-half
of zeros. Thus we lift the point $(C_\fm,{}^t\!D_\fm)$ to the same named point in $\scrX^\vee_{M,N}$.
The equality $^t\!D_\fm C_\fm=0=C_\fm{}^t\!D_\fm$ implies that $(C_\fm,{}^t\!D_\fm)\in\Lambda_{\scrX_{M,N}^\vee}$.
It follows from~Lemma~\ref{relev} that the (same named) image of $(C_\fm,{}^t\!D_\fm)$ in
$\scrQ_{M,N}$ lies in a relevant orbit.

(b) follows e.g.\ from~Lemma~\ref{pi MNN} and the row reduction. 

(c) follows from (b) since we know from~Lemma~\ref{relev} and~Lemma~\ref{pi MNN} that the
number of relevant orbits is equal to the number of mazes.

(d) Similar to the above.
\end{proof}

\subsection{Equivariant slice, I}
\label{slice GL}
Now we are in a position to prove~Proposition~\ref{gl(M|N) lagr}.

The irreducible components of $\Lambda_{\scrX_{M,N}}$ are the twisted conormal bundles to the
{\em relevant} $B$-orbits in $\scrY_{M,N}$. Equivalently, we consider the flag variety $\CB=G/B$
with the action of $H_{M,N}$. Then an $H_{M,N}$-orbit in $\CB$ is called relevant if the restriction of
$\psi$ to the stabilizer of a point in this orbit is trivial. Let us give yet one more equivalent
but more economical reformulation. Let $B_M\subset\GL(M),\ B_N\subset\GL(N)$
be the upper-triangular Borel subgroups. Let $H^M_N:=B_M\ltimes\on{U}\subset\GL(N)$.
Then $H^M_N$ acts on the flag variety $\CB_N=\GL(N)/B_N$ with finitely many orbits, and we have to count
the relevant $H^M_N$-orbits.

Let $B^M_N$ denote the stabilizer of the basis vector $v_{M+1}$ in the upper-triangular Borel subgroup
$B_N\subset\GL(N)$. Clearly, $H^M_N\subset B^M_N$, so we have the natural map of orbit sets
$\alpha\colon H^M_N\backslash\CB_N\to B^M_N\backslash\CB_N$. We have an exact sequence of groups
\[1\to H^M_N\to B^M_N\to\BG_m^{N-M-1}\to1.\] Since $H^M_N\backslash\CB_N$ is finite, and $\BG_m^{N-M-1}$
is connected, $\alpha$ is a bijection.

On the other hand, the set $B^M_N\backslash\CB_N$ is in a natural bijection with
set of $B_N$-orbits in the subset $\scrM^M_N\subset\CB_N\times\BC^N$ formed by the pairs $(F_\bullet,v)$
such that $v\in F_{M+1}\setminus F_M$. According to~\cite[Lemma 2]{t}, the orbits $B_N\backslash\scrM^M_N$
are indexed by colored permutations $(w,\beta)\in RB^M_N$ (notation introduced
after~Corollary~\ref{symp mirab}). 
Recall from~\cite[Lemma 2]{t} that the set $RB_N$ of colored permutations is in a natural bijection
with the set of pairs $(w,\sigma)$ where $w$ is a permutation of $\{1,\ldots,N\}$, and $\sigma$
is a non-empty decreasing sequence. 

Now the subset of relevant orbits among $H^M_N\backslash\CB_N$ corresponds to the subset
$RB_{M,N}\subset RB^M_N$ (notation introduced after~Corollary~\ref{symp mirab}). Indeed, we choose the
following representatives of the double cosets
$B^M_N\backslash\GL(N)/B_N$. Given $(w,\beta)\in RB^M_N$ we consider the corresponding decreasing
sequence $(w,\sigma)$. The desired representative $A_{(w,\beta)}\in\GL(N)$ is the sum of two matrices
$A_{(w,\beta)}=A_1+A_2$, where $A_1$ is the matrix of $w$ (that is $a_{ij}=1$ iff $j=w(i)$),
and the entries of $A_2$ all vanish except for $a_{M+1,K}=1$ iff $K\leq M$ and $K$ belongs to $\sigma$.
Since the corresponding points in $\CB_N=\GL(N)/B_N$ are all fixed by the diagonal subgroup
$\BG_m^{N-M-1}\subset B^M_N$, the restriction of the character
$\psi=u_{M+1,M+2}+\ldots+u_{N-1,N}\colon\on{U}_{M,N}\to\BG_a$ to the stabilizer of a point vanishes iff
all the summands vanish. And the summand $u_{M+1+k,M+2+k}$ vanishes iff $w(M+1+k)>w(M+2+k)$.

It remains to apply~Lemma~\ref{rho}.
\hfill $\Box$

\bigskip

\noindent
The rest of~Section~\ref{glmn} lays ground for the study of orthosymplectic case in~\S\ref{osp}.

\subsection{Equivariant slice, II}
\label{slice GL'}
More generally, for $M<N$, the coisotropic variety $\scrX_{M,N}$ is also isomorphic to a twisted
cotangent bundle of the following spherical $G$-variety $\scrY_{M,N}^{(r,s,k)}$. For
$r+s=N-M-1,\ r,s\geq0$, let $\on{U}_{M,N}^{(r,s)}\subset\GL(N)$ be the unipotent subgroup
$\begin{pmatrix}U_r& {*} &{*}\\0&1_{M+1}& {*}\\0&0& U_s\end{pmatrix}$ of~\cite[Remark 2.4.1]{bfgt}.
We also fix an integer $k$ such that $r+1\leq k\leq N-s$. It gives rise to an embedding of
$\GL(M)$ into $\begin{pmatrix}1_r& 0 &0\\0&\GL(M+1)& 0\\0&0& 1_s\end{pmatrix}$ as the set of
matrices with the $k$-th row and the $k$-th column having all 0's except for~1 at $k$-th position.
We set $H_{M,N}^{(r,s,k)}:=\GL(M)\ltimes\on{U}_{M,N}^{(r,s)}$. We denote the character of
$\on{U}_{M,N}^{(r,s)}$ corresponding to the character
$\chi_{M,N}^{(r,s)}\colon\on{Lie}\on{U}_{M,N}^{(r,s)}\to\BG_a,\
(u_{ij})\mapsto\sum_{i=1}^{r-1}u_{i,i+1}+u_{rk}+u_{k,N-s+1}+\sum_{i=N-s+1}^{N-1}u_{i,i+1}$,
of {\em loc.cit.}\ by $\psi$ (note that it actually depends on the choice of $k$). This character
extends to the same named character $\psi\colon H_{M,N}^{(r,s,k)}\to\BG_a$. Then
$\scrY_{M,N}^{(r,s,k)}=\GL(N)/\on{U}_{M,N}^{(r,s,k)}=G/H_{M,N}^{(r,s,k)}$, and
$\scrX_{M,N}\simeq T^*_\psi\scrY_{M,N}^{(r,s,k)}$, cf.~\cite[\S5.3]{ty1}.

For the orthosymplectic applications we will need a bijection between $\fM_{M,N}$ and the set of
relevant $B$-orbits in $\scrY_{M,N}^{(r,s,k)}$. As in the beginning of~\S\ref{slice GL}, the latter
set is in a natural bijection with the set of relevant $H_{M,N}^{(r,s,k)}$-orbits in $\CB$,
or else with the set of relevant $H^M_{N,r,s,k}$-orbits in $\CB_N$. Here
$H^M_{N,r,s,k}:=B_M\ltimes\on{U}_{M,N}^{(r,s,k)}\subset\GL(N)$, and $B_M$ is the upper triangular
Borel subgroup in $\GL(M)$ embedded into $\GL(N)$ as in the previous paragraph (this embedding
depends on $r,s,k$).

We list the representatives in $\GL(N)$ of the relevant double cosets
$H^M_{N,r,s,k}\backslash\GL(N)/B_N$. We embed the $M\times N$-chessboard into the $N\times N$-chessboard
as follows. We add $r$ rows at the top, then~1 row at the position $k$, and then $s$ rows at the
bottom. A maze $\fm\in\fM_{M,N}$ has $r+s+1$ walls of type~III. Ignoring for a moment the $k$-th
row in the middle, we extend the rightmost $r$ (resp.\ the leftmost $s$) type~III walls of $\fm$ to
the walls of type~II (resp.\ of type~I) in the $(N-1)\times N$-chessboard (such an extension is
uniquely defined). We also extend the remaining type~III wall (resp.\ all the type~I,\ resp.\ all the
type~II walls) of $\fm$ to the type~III
wall $W_{III}$ (resp.\ to the type~I, resp.\ to the type~II walls) of the resulting
$(N-1)\times N$-maze (such an extension is uniquely defined).
Now we restitute the $k$-th row in the middle to obtain the $N\times N$-chessboard. We place a
rook in the box of the $k$-th row adjacent to the wall $W_{III}$,
and we denote by $i$ the number of this box in the $k$-th row.
We also place the rooks into all the boxes on the wall $W_{III}$ marked with x above the $k$-th row
and marked with o below the $k$-th row. We also place the rooks into all
the boxes marked with x to the left of $W_{III}$ and into all the boxes marked with o to the right
of $W_{III}$. The resulting rooks placement corresponds to a permutation $w'$ of $\{1,\ldots,N\}$.
We set $w=w_0^{(r)}w_0^{(s)}w'$, where $w_0^{(r)}$ (resp.\ $w_0^{(s)}$) is the order-reversing
permutation of $\{1,\ldots,r\}$ (resp.\ of $\{N-s+1,\ldots,N\}$).
Finally, the desired representative $A_\fm^{r,s,k}\in\GL(N)$ has 1's at the graph of $w$, and also in the
$k$-th row in all the columns containing the corners of $W_{III}$
below the $k$-th row, and also in
the $i$-th column in all the rows containing the corners of $W_{III}$ above the $k$-th row,
and 0's everywhere else.

\subsection{Independence of $r,s,k$}
\label{indep rsk}
Under the identification $\scrX_{M,N}\cong T^*_\psi\scrY_{M,N}^{(r,s,k)}$, the irreducible components of
$\Lambda_{\scrX_{M,N}}$ are the closures of the twisted conormal bundles of the relevant $B$-orbits in
$\scrY_{M,N}^{(r,s,k)}$. Thus the set $\on{Irr}\Lambda_{\scrX_{M,N}}$ of irreducible components of
$\Lambda_{\scrX_{M,N}}$ is in a natural bijection with the set of relevant $B$-orbits in $\scrY_{M,N}^{(r,s,k)}$,
or equivalently, with the set of relevant $H^M_{N,r,s,k}$-orbits in $\CB_N$. Hence the bijection
between the latter set and $\fM_{M,N}$ constructed in~\S\ref{slice GL'} gives rise to the bijection
$\theta^M_{N,r,s,k}\colon\fM_{M,N}\iso\on{Irr}\Lambda_{\scrX_{M,N}}$.

\begin{prop}
  \label{independence}
  We have $\theta^M_{N,r,s,k}=\theta^M_{N,r',s',k'}\colon\fM_{M,N}\iso\on{Irr}\Lambda_{\scrX_{M,N}}$.
  The common value of all these bijections will be denoted
  $\theta_{M,N}\colon\fM_{M,N}\iso\on{Irr}\Lambda_{\scrX_{M,N}}$.
\end{prop}

\begin{proof}
  First we keep $r,s$ intact.

  \begin{lem}
    \label{indep k}
    We have $\theta^M_{N,r,s,k}=\theta^M_{N,r,s,k'}\colon\fM_{M,N}\iso\on{Irr}\Lambda_{\scrX_{M,N}}$.
  \end{lem}

  \begin{proof}
We assume $k'>k$. We consider the cyclic
permutation $w_{k,k'}\colon k\mapsto k+1\mapsto\ldots\mapsto k'\mapsto k$ and view it as a
permutation of $\{1,\ldots,N\}$ fixing all the elements less than $k$ or bigger than $k'$.
We have to check that $w_{k,k'}A_\fm^{r,s,k}$ and $A_\fm^{r,s,k'}$ (representatives constructed
in~\S\ref{slice GL'}) lie in the same double coset in $H^M_{N,r,s,k'}\backslash\GL(N)/B_N$.
These two representatives coincide away from all the columns containing the vertical segments of
the wall $W_{III}$ (notation of {\em loc.cit.}) and all the rows containing the horizontal segments
of the wall $W_{III}$ plus the $k$-th row. These columns and rows form the square
$l\times l$-submatrices $A,A'$ (where $l$ stands for the number of vertical segments of $W_{III}$),
and the remaining columns and rows form the square
$(N-l)\times(N-l)$-submatrix $C$ common for $w_{k,k'}A_\fm^{r,s,k}$ and $A_\fm^{r,s,k'}$.
Moreover, $w_{k,k'}A_\fm^{r,s,k}=A\oplus C$ and $A_\fm^{r,s,k'}=A'\oplus C$ (so that the matrices of
our representatives are block-diagonal up to permutation of rows and columns).

We consider the Levi subgroup $\GL(l)\times\GL(N-l)\subset\GL(N)$. The intersection of
$B_N\subset\GL(N)$ (acting on the right) with $\GL(l)$ is $B_l$. The intersection of
$H^M_{N,r,s,k'}\subset\GL(N)$ (acting on the left) with $\GL(l)$ is $B_{l-1}$: the upper triangular
subgroup in $\GL(l-1)$ embedded into $\GL(l)$ as the stabilizer of the $k'$-th vector of the
standard basis, and of the $k'$-th covector of the standard basis. So it is enough to verify that
both $A$ and $A'$ lie in a common double coset in $B_{l-1}\backslash\GL(l)/B_l$. But it is
straightforward to see that both $A,A'$ lie in the open double coset.
  \end{proof}
  
Provided independence of $k$ (for fixed $r,s$) we have just established, in order to complete
the proof of~Proposition~\ref{indep rsk}, it remains to prove the following

\begin{lem}
  \label{indep rs}
We have $\theta^M_{N,r,s,r+1}=\theta^M_{N,r+1,s-1,N-s+1}\colon\fM_{M,N}\iso\on{Irr}\Lambda_{\scrX_{M,N}}$.
\end{lem}

\begin{proof}
To compare the relevant $B_M\times B_N$-orbits in $\scrY_{M,N}^{(r,s,r+1)}=\on{U}^{(r,s)}_{M,N}\backslash\GL(N)$
and in $\scrY_{M,N}^{(r+1,s-1,N-s+1)}=\on{U}^{(r+1,s-1)}_{M,N}\backslash\GL(N)$ we consider the projections
\[\scrY_{M,N}^{(r,s,r+1)}\xrightarrow{p}\scrZ=
\on{U}^{(r+1,s)}_{M-1,N}\backslash\GL(N)\xleftarrow{q}\scrY_{M,N}^{(r+1,s-1,N-s+1)}.\]
Note that $p$ and $q$ are dual vector bundles of rank $M$.
The double cosets in $B_M\ltimes\on{U}^{(r+1,s)}_{M-1,N}\backslash\GL(N)/B_N$ are naturally indexed by the
set of permutations of $\{1,\ldots,N\}$, and both representatives $A_\fm^{r,s,r+1},A_\fm^{r+1,s-1,N-s+1}$
project to the same $B_M\times B_N$-orbit $\BO_w\subset\scrZ$ (the permutation $w$ was defined
in~\S\ref{slice GL'}). Let us denote the orbit $B_M\cdot A_\fm^{r,s,r+1}\cdot B_N$ (resp.\
$B_M\cdot A_\fm^{r+1,s-1,N-s+1}\cdot B_N$) by $\BO_\fm^{r,s,r+1}\subset\scrY_{M,N}^{(r,s,r+1)}$ (resp.\ by
$\BO_\fm^{r+1,s-1,N-s+1}\subset\scrY_{M,N}^{(r+1,s-1,N-s+1)}$).

We choose the basis vectors in $\BC^M$ with indices equal to the column numbers of all boxes
marked with x on $W_{III}$ (notation of~\S\ref{slice GL'}) and to the left of $W_{III}$.
They span a vector subspace $V\subset\BC^M$, and the remaining dual basis vectors span a vector subspace
$V^\perp\subset\BC^{M*}$. The fiber $p^{-1}(w)$ (resp.\ $q^{-1}(w)$) over $w\in\BO_w$ is an open subspace
in $V$ (resp.\ in $V^\perp$). Hence the twisted conormal bundles
$T^*_{\BO_\fm^{r,s,r+1}}\scrY_{M,N}^{(r,s,r+1)}\subset T^*_\psi\scrY_{M,N}^{(r,s,r+1)}\cong\scrX_{M,N}$ and
$T^*_{\BO_\fm^{r+1,s-1,N-s+1}}\scrY_{M,N}^{(r+1,s-1,N-s+1)}\subset T^*_\psi\scrY_{M,N}^{(r+1,s-1,N-s+1)}\cong\scrX_{M,N}$
have the same closure and lie in the same irreducible component in $\on{Irr}\Lambda_{\scrX_{M,N}}$.
\end{proof}  

This completes the proof of~Proposition~\ref{indep rsk}.
\end{proof}

\subsection{Involutions}
\label{invol}
For $M,N$ of opposite parity, we consider the involution
$\iota\colon\GL(N)\to\GL(N),\ g\mapsto w_0{}^t\!g^{-1}w_0$, where $w_0\in\GL(N)$ is the matrix with
1's at the antidiagonal, and 0's everywhere else. Equivalently, the matrix $\iota(g)$ is the inverse of
the matrix $g$ transposed with respect to the antidiagonal. The fixed point set $\GL(N)^\iota$ is the
group $\on{O}(N)$ of orthogonal transformations preserving the symmetric bilinear form given by the
matrix with 1's at the antidiagonal and 0's everywhere else. We have $\iota(B_N)=B_N$, and
$\iota(H^M_{N,r,s,k})=H^M_{N,s,r,N+1-k}$.

If $M,N$ are both even, let $w'_0$ stand for the $N\times N$-matrix with 1's in the right upper half of
the antidiagonal, $-1$'s in the left lower half of the antidiagonal, and 0's everywhere else.
We consider the involution $\iota\colon\GL(N)\to\GL(N),\ g\mapsto -w'_0{}^t\!g^{-1}w'_0$.
The fixed point set $\GL(N)^\iota$ is the
group $\Sp(N)$ of symplectic transformations preserving the skew symmetric bilinear form given by the
matrix $w'_0$. We have $\iota(B_N)=B_N$, and $\iota(H^M_{N,r,s,k})=H^M_{N,s,r,N+1-k}$.

Hence $\iota$ induces the same named involution
\[\iota\colon\scrX_{M,N}\cong T^*_\psi\scrY_{M,N}^{(r,s,k)}\to T^*_\psi\scrY_{M,N}^{(s,r,N+1-k)}\cong\scrX_{M,N},\]
and also the same named involution $\iota\colon\on{Irr}\Lambda_{\scrX_{M,N}}\to\on{Irr}\Lambda_{\scrX_{M,N}}$.

\begin{cor}
  \label{iota slice}
  We have $\iota\theta_{M,N}=\theta_{M,N}\iota$.
\end{cor}

\begin{proof}
It is straightforward to check that the representatives $\iota(A_\fm^{r,s,k})$ and
$A_{\iota\fm}^{s,r,N+1-k}$ lie in the same $H^M_{N,s,r,N+1-k}$-orbit.
\end{proof}

\subsection{Equivariant slice, III}
In case $N=2n$ is even, we will need two more `exotic' realizations of $\scrX_{M,N}$.

\subsubsection{}
\label{slice GL''}
In case $M=2m+1<2n=N$, we will use one more realization of $\scrX_{M,N}\simeq T^*_\psi\scrY_{M,N}^{\on{mid}}$.
Namely, we set $r=s=n-m-1$, and we embed $\GL(M)$ into
$\begin{pmatrix}1_{n-m-1}& 0 &0\\0&\GL(M+1)& 0\\0&0& 1_{n-m-1}\end{pmatrix}$ as the set of matrices
preserving (in the standard basis $v_1,\ldots,v_N$ of $\BC^N$) the vectors
$v_1,\ldots,v_{n-m-1},v_n-v_{n+1},v_{n+m+2},\ldots,v_N$ and (in the dual basis $v_1^*,\ldots,v_N^*$ of
$\BC^{N*}$) the covectors $v_1^*,\ldots,v_{n-m-1}^*,v_n^*-v_{n+1}^*,v_{n+m+2}^*,\ldots,v_N^*$.
In other words, the central $2\times2$-block consists of matrices
$\begin{pmatrix}a+\frac12&a-\frac12\\ a-\frac12&a+\frac12\end{pmatrix},\ a\ne0$,
and apart from this block,
the two middle columns coincide, and the two middle rows coincide. We set
$H^{\on{mid}}_{M,N}:=\GL(M)\ltimes\on{U}_{M,N}^{(r,s)}$ (recall that $r=s=n-m-1$). We denote by $\psi$ the
character of $\on{U}_{M,N}^{(r,s)}$ corresponding to the character
$\chi_{M,N}^{\on{mid}}\colon\on{Lie}\on{U}_{M,N}^{(r,s)}\to\BG_a,\ (u_{ij})\mapsto\sum_{i=1}^{n-m-2}u_{i,i+1}+
u_{n-m-1,n}+u_{n-m-1,n+1}+u_{n,n+m+2}+u_{n+1,n+m+2}+\sum_{i=n+m+2}^{N-1}u_{i,i+1}$. This character
extends to the same named character $\psi\colon H_{M,N}^{\on{mid}}\to\BG_a$. Then
$\scrY_{M,N}^{\on{mid}}=\GL(N)/\on{U}_{M,N}^{(r,s)}=G/H_{M,N}^{\on{mid}}$, and
$\scrX_{M,N}\simeq T^*_\psi\scrY_{M,N}^{\on{mid}}$.

Since the irreducible components of $\Lambda_{\scrX_{M,N}}$ are the closures of twisted conormal bundles
of the relevant $B$-orbits in $\scrY_{M,N}^{\on{mid}}$, the set $\on{Irr}\Lambda_{\scrX_{M,N}}$ is in a natural
bijection with the set of relevant $B$-orbits in $\scrY_{M,N}^{\on{mid}}$, or equivalently, with the set of
relevant $H^M_{N,\on{mid}}$-orbits in $\CB_N$. Here
$H^M_{N,\on{mid}}:=B_M\ltimes\on{U}_{M,N}^{(r,s)}\subset\GL(N)$, and $B_M$ is the upper triangular
Borel subgroup in $\GL(M)$ embedded into $\GL(N)$ as in the previous paragraph.
Composing with the bijection $\theta_{M,N}$ of~Proposition~\ref{independence}, we obtain the same named bijection
$\theta_{M,N}$ from $\fM_{M,N}$ to the set of relevant $H^M_{N,\on{mid}}$-orbits in $\CB_N$.

\subsubsection{}
\label{slice GL'''}
In case $M=2m<2n=N$ are both even, we will use one more realization of
$\scrX_{M,N}\simeq T^*_\psi\scrY_{M,N}^{\on{mid}}$. Namely, let $\on{U}_{M,N}^{\on{mid}}\subset\GL(N)$ be the
unipotent subgroup $\begin{pmatrix}U_{n-m-1}& {*} &{*}\\0&\CH_M& {*}\\0&0& U_{n-m-1}\end{pmatrix}$,
where $\CH_M\subset\GL(M+2)$ is the $2m+1$-dimensional commutative subgroup of strictly upper
triangular matrices (with all 1's at the diagonal) with arbitrary $m+1$ last entries in the first row,
arbitrary $m+1$ first entries in the last column, and 0's everywhere else.
Let also $\GL(M)\subset\GL(M+2)$ be the central block embedding (so the first and last row and column
vanish except for 1's at the diagonal), and let $B_M\subset\GL(M)\subset\GL(M+2)$ be the composed
embedding. We set $H^M_{N,\on{mid}}:=B_M\ltimes\on{U}_{M,N}^{\on{mid}}\subset\GL(N)$.
We have the additive character $\psi\colon H^M_{N,\on{mid}}\to\BG_a$ corresponding to the character
$\chi_{M,N}^{\on{mid}},\ (u_{ij})\mapsto\sum_{i=1}^{n-m-1}u_{i,i+1}+u_{n-m,n+m+1}+\sum_{i=n+m+1}^{N-1}u_{i,i+1}$,
of its Lie algebra. Then $\scrY_{M,N}^{\on{mid}}=\GL(N)/\on{U}_{M,N}^{\on{mid}}$, and
$\scrX_{M,N}\simeq T^*_\psi\scrY_{M,N}^{\on{mid}}$, cf.~\cite[\S3.2]{bft2}. Note that $\scrY_{M,N}^{\on{mid}}$ does
not carry an action of $\GL(M)$, but does carry an action of $B_M$ (commuting with the natural action
of $\GL(N)$). However, the isomorphism $\scrX_{M,N}\simeq T^*_\psi\scrY_{M,N}^{\on{mid}}$ is {\em not}
$B_M$-equivariant.

Since the irreducible components of $\Lambda_{\scrX_{M,N}}$ are the closures of twisted conormal bundles
of the relevant $B$-orbits in $\scrY_{M,N}^{\on{mid}}$, the set $\on{Irr}\Lambda_{\scrX_{M,N}}$ is in a natural
bijection with the set of relevant $B$-orbits in $\scrY_{M,N}^{\on{mid}}$, or equivalently, with the set of
relevant $H^M_{N,\on{mid}}$-orbits in $\CB_N$. Composing with the bijection $\theta_{M,N}$
of~Proposition~\ref{independence}, we obtain the same named bijection
$\theta_{M,N}$ from $\fM_{M,N}$ to the set of relevant $H^M_{N,\on{mid}}$-orbits in $\CB_N$.

\begin{rem}
  We list the representatives in $\GL(N)$ of the relevant double cosets
$H^M_{N,\on{mid}}\backslash\GL(N)/B_N$. We embed the $M\times N$-chessboard into the $N\times N$-chessboard
adding $n-m$ rows at the top and $n-m$ rows at the bottom. A maze $\fm\in\fM_{M,N}$ has $N-M=2n-2m$
walls of type~III. We extend the rightmost (resp.\ leftmost) $n-m$ type~III walls to the walls of
type~II (resp.\ type~I) in the $N\times N$-chessboard. In the resulting $N\times N$-maze we
place the rooks into all the boxes marked with x on the extended $(n-m)$-th type~III wall and to
the left of it. We also place the rooks into all the boxes marked with o on the extended $(n-m+1)$-th
type~III wall and to the right of it. The resulting rooks placement corresponds to a permutation
$w'$ of $\{1,\ldots,N\}$. We set $w=w_0^{(u)}w_0^{(d)}w'$, where $w_0^{(u)}$ (resp.\ $w_0^{(d)}$)
is the order-reversing permutation of $\{1,\ldots,n-m\}$ (resp.\ of $\{n+m+1,\ldots,N\}$).
Now we construct an auxiliary {\em hybrid} wall $W_h$ as follows. It coincides with the extended
$(n-m)$-th type~III wall below the middle line (dividing our $N\times N$-chessboard into the upper and
lower $n\times N$-chessboards), then follows the middle line till its intersection with the extended
$(n-m+1)$-th type~III wall, and then goes up along the latter wall. The desired representative
$E_\fm\in\GL(N)$ has 1's at the graph of $w$, and also in the $(n-m)$-th (resp.\ $(n+m+1)$-th) row in
all the columns containing the corners of $W_h$ above (resp.\ below) its intersection with the middle
line, including the rightmost (resp.\ leftmost) intersection with the middle line.
\end{rem}

\subsection{Fourier transform}
\label{fourier}
Recall that the set of $B^\vee$-orbits in $\Hom(\BC^N,\BC^M)$ is naturally indexed by $\fR_{M,N}$.
The set of $B^\vee$-orbits in $\Hom(\BC^M,\BC^N)$ is also naturally indexed by $\fR_{M,N}$.
Given $\frr\in \fR_{M,N}$ we denote by $\Lambda_\frr\subset\scrX_{M,N}^\vee$
(resp.\ $\Lambda'_\frr\subset\scrX_{M,N}^\vee$) the closure of the
conormal bundle to the corresponding $B^\vee$-orbit in $\Hom(\BC^N,\BC^M)$
(resp.\ in $\Hom(\BC^M,\BC^N)$).

\begin{lem}
  $\Lambda'_\frr=\Lambda_{\frr'}$, where $\frr'=\pi_{M,N}\iota\pi_{M,N}^{-1}(\frr)$ (notation
  of~Lemma~\ref{pi MN}).
\end{lem}

\begin{proof}
A point of $\Lambda_\frr$ is obtained
as follows. We consider the maze $\pi_{M,N}^{-1}(\frr)$ and take the matrix
$C_{\frr}\in\Hom(\BC^N,\BC^M)$ with 1's at the positions marked by x's and 0's everywhere else.
We also take the matrix $D_{\frr}\in\Hom(\BC^N,\BC^M)$ with 1's at the positions marked by o's and
0's everywhere else. Then $(w_0C_{\frr},{}^t\!D_{\frr}w_0)\in\Lambda_\frr$
(here $w_0$ is the $M\times M$-matrix with 1's at the antidiagonal and 0's everywhere else,
and ${}^t\!D_{\frr}\in\Hom(\BC^M,\BC^N)$ is the transposed of $D_{\frr}$). Indeed, by the definition
of a maze, both compositions $w_0C_{\frr}{}^t\!D_{\frr}w_0$ and
${}^t\!D_{\frr}w_0w_0C_{\frr}={}^t\!D_{\frr}C_{\frr}$ are strictly upper-triangular. And
$\bmu^\svee(w_0C_{\frr},{}^t\!D_{\frr}w_0)=({}^t\!D_{\frr}C_{\frr},w_0C_{\frr}{}^t\!D_{\frr}w_0)$.

By definition of $\iota$, we have $C_{\frr'}=D_{\frr},\ D_{\frr'}=C_{\frr}$.
Thus our point $(w_0C_{\frr},{}^t\!D_{\frr}w_0)$ lies in $\Lambda_{\frr'}$ as well.
It follows that $\Lambda'_\frr=\Lambda_{\frr'}$. Indeed, let us define
$\frr''$ by the property  $\Lambda'_\frr=\Lambda_{\frr''}$. Then we have to
prove $\frr''=\frr'$. But both maps $\frr\mapsto\frr'$ and
$\frr\mapsto\frr''$ are bijections $\fR_{M,N}\iso \fR_{M,N}$. Hence
$\frr'\mapsto\frr''$ is an automorphism of the set of $B^\vee$-orbits in $\Hom(\BC^M,\BC^N)$.
Also, by definition, the $B^\vee$-orbit
in $\Hom(\BC^M,\BC^N)$ corresponding to $\frr'$ lies in the closure of the $B^\vee$-orbit
in $\Hom(\BC^M,\BC^N)$ corresponding to $\frr'$. It follows $\frr''=\frr'$.
The lemma is proved.
\end{proof}

\section{Orthosymplectic}
\label{osp}

\subsection{Setup}
\label{osp setup}
The action of $G_{2n,2n+1}:=\SO(2n)\times\SO(2n+1)$ (resp.\ of $G_{2n-1,2n}:=\SO(2n-1)\times\SO(2n)$)
on $\scrX^\iota_{2n,2n+1}:=T^*\SO(2n+1)$ (resp.\ on $\scrX^\iota_{2n-1,2n}:=T^*\SO(2n)$) is coisotropic.
More generally, let $e\in\fso(2n+1)$ be a nilpotent element of Jordan type $(2n+1-2m,1^{2m})$, and
let $\CS_e$ be a Slodowy slice to the nilpotent orbit $\BO_e\subset\fso(2n+1)$.
The centralizer of $e$ in $\SO(2n+1)$ contains $\SO(2m)$, and by the construction of~\cite[\S2]{l},
cf.~\cite[\S3.4]{bzsv}, the product $\scrX^\iota_{2m,2n+1}:=\SO(2n+1)\times\CS_e$ carries a symplectic
structure and a hamiltonian action of $G_{2m,2n+1}:=\SO(2m)\times\SO(2n+1)$. The action
$G_{2m,2n+1}\circlearrowright\scrX^\iota_{2m,2n+1}$ is coisotropic, see e.g.~\cite[\S2.2]{fu}.
It satisfies the conditions~\cite[\S3.5.1]{bzsv} of hypersphericity. Moreover,
$\scrX^\iota_{2m,2n+1}$ is polarizable, i.e.\ $\scrX^\iota_{2m,2n+1}$ is a twisted cotangent bundle of the
following spherical $G_{2m,2n+1}$-variety $\scrY^\iota_{2m,2n+1}$. Recall the unipotent subgroup
$\on{U}_{2m,2n+1}^{(r,s)}\subset\GL(N)$ of~\cite[Remark 2.4.1]{bfgt}, where $r=s=n-m$.
Let $\iota$ be the involution of $\GL(2n+1)$ preserving the upper triangular Borel subgroup,
such that the fixed point set is the orthogonal subgroup $\on{O}(2n+1)$ preserving the symmetric
bilinear form whose matrix consists of 1's at the antidiagonal, see~\S\ref{invol}. Then
$\scrY^\iota_{2m,2n+1}=\SO(2n+1)/(\on{U}_{2m,2n+1}^{(r,s)})^\iota$ (quotient modulo the fixed point set in the
unipotent subgroup).

Let $e\in\fso(2n)$ be a nilpotent element of Jordan type $(2n-1-2m,1^{2m+1})$, and
let $\CS_e$ be a Slodowy slice to the nilpotent orbit $\BO_e\subset\fso(2n)$.
The centralizer of $e$ in $\SO(2n)$ contains $\SO(2m+1)$, and by the construction of~\cite[\S2]{l},
cf.~\cite[\S3.4]{bzsv}, the product $\scrX^\iota_{2m+1,2n}:=\SO(2n)\times\CS_e$ carries a symplectic
structure and a hamiltonian action of $G_{2m+1,2n}:=\SO(2m+1)\times\SO(2n)$. The action
$G_{2m+1,2n}\circlearrowright\scrX^\iota_{2m+1,2n}$ is coisotropic, see e.g.~\cite[\S2.2]{fu}.
It satisfies the conditions~\cite[\S3.5.1]{bzsv} of hypersphericity. Moreover,
$\scrX^\iota_{2m+1,2n}$ is polarizable, i.e.\ $\scrX^\iota_{2m+1,2n}$ is a twisted cotangent bundle of the
following spherical $G_{2m+1,2n}$-variety $\scrY^\iota_{2m+1,2n}$. Recall the unipotent subgroup
$\on{U}_{2m+1,2n}^{(r,s)}\subset\GL(N)$ of~\cite[Remark 2.4.1]{bfgt}, where $r=s=n-m-1$.
Let $\iota$ be the involution of $\GL(2n)$ preserving the upper triangular Borel subgroup,
such that the fixed point set is the orthogonal subgroup $\on{O}(2n)$ preserving the symmetric
bilinear form whose matrix consists of 1's at the antidiagonal, see~\S\ref{invol}. Then
$\scrY^\iota_{2m+1,2n}=\SO(2n)/(\on{U}_{2m+1,2n}^{(r,s)})^\iota$ (quotient modulo the fixed point set in the
unipotent subgroup).

The recipe of~\cite[\S4]{bzsv} produces from $G_{2m,2n+1}\circlearrowright\scrX^\iota_{2m,2n+1}$
(resp.\ $G_{2m+1,2n}\circlearrowright\scrX^\iota_{2m+1,2n}$) the dual symplectic variety
$\scrX^{\iota\vee}_{2m,2n+1}$ (resp.\ $\scrX^{\iota\vee}_{2m+1,2n}$) with a hamiltonian coisotropic action
of $G_{2m,2n+1}^\vee=\SO(2m)\times\Sp(2n)$ (resp.\ $G_{2m+1,2n}^\vee=\Sp(2m)\times\SO(2n)$).
Namely, $\scrX^{\iota\vee}_{2m,2n+1}$ (resp.\ $\scrX^{\iota\vee}_{2m+1,2n}$) is a symplectic vector representation
of $G_{2m,2n+1}^\vee$ (resp.\ of $G_{2m+1,2n}^\vee$) equal to $\BC^{2m}_+\otimes\BC^{2n}_-$
(resp.\ $\BC^{2m}_-\otimes\BC^{2n}_+$).

Furthermore, the action of $G_{2n,2n}:=\Sp(2n)\times\Sp(2n)$ on
$\scrX^\iota_{2n,2n}:=(T^*\Sp(2n))\times\BC^{2n}$ is coisotropic. More generally, let $e\in\fsp(2n)$
be a nilpotent element of Jordan type $(2n-2m,1^{2m})$, and let $\CS_e$ be a Slodowy slice to the
nilpotent orbit $\BO_e\subset\fsp(2n)$. The centralizer of $e$ in $\Sp(2n)$ contains $\Sp(2m)$
and by the construction of~\cite[\S2]{l}, cf.~\cite[\S3.4]{bzsv}, the product
$\scrX^\iota_{2m,2n}:=\Sp(2n)\times\CS_e$ carries a symplectic structure and a hamiltonian action of
$G_{2m,2n}:=\Sp(2m)\times\Sp(2n)$. The action $G_{2m,2n}\circlearrowright\scrX^\iota_{2m,2n}$ is
coisotropic, see e.g.~\cite[\S2.2]{fu}. It satisfies the conditions~\cite[\S3.5.1]{bzsv}
of hypersphericity, but it is not polarizable, and the recipe of~\cite[\S4]{bzsv} is inapplicable due
to the nontrivial anomaly. Still it is expected that $\scrX^\iota_{2m,2n}$
has the dual symplectic variety $\scrX^{\iota\vee}_{2m,2n}$ (resp.\ $'\scrX^{\iota\vee}_{2m,2n}$) with a
hamiltonian coisotropic action of $G^\vee_{2m,2n}=\SO(2m+1)\times\Sp(2n)$ (resp.\
$'G^\vee_{2m,2n}=\Sp(2m)\times\SO(2n+1)$). Namely, $\scrX^{\iota\vee}_{2m,2n}$ (resp.\ $'\scrX^{\iota\vee}_{2m,2n}$)
is a symplectic vector representation of $G^\vee_{2m,2n}$ (resp.\ of $'G^\vee_{2m,2n}$) equal to
$\BC^{2m+1}_+\otimes\BC^{2n}_-$ (resp.\ $\BC^{2m}_-\otimes\BC^{2n+1}_+$).

By definition, our symplectic groups preserve the skew symmetric bilinear forms with matrices
vanishing off the antidiagonal, and 1's (resp.\ $-1$'s) in the upper right half (resp.\ lower left
half) of the antidiagonal. Thus we realize the symplectic groups as the fixed point sets of
involutions (also called $\iota$) of the ambient general linear groups preserving the upper
triangular Borel subgroups, see~\S\ref{invol}. We choose the upper triangular Borel subgroups in our
symplectic groups (as well as in our special orthogonal groups). Hence we obtain the subvarieties
$\Lambda_{\scrX_{2m+1,2n}^\iota},\Lambda_{'\!\scrX^{\iota\vee}_{2m,2n}}$, etc.

\begin{prop}
  \label{osp lagr}
  \textup{(a)} Both $\Lambda_{\scrX_{2m,2n+1}^\iota}$ and $\Lambda_{\scrX_{2m,2n+1}^{\iota\vee}}$ are Lagrangian
  subvarieties. The number of their irreducible components is equal to the number of centrally
  symmetric $2m\times2n$-mazes.

  \textup{(b)} Both $\Lambda_{\scrX_{2m+1,2n}^\iota}$ and $\Lambda_{\scrX_{2m+1,2n}^{\iota\vee}}$ are Lagrangian
  subvarieties. The number of their irreducible components is equal to the number of centrally
  symmetric $2m\times2n$-mazes.

  \textup{(c)} $\Lambda_{\scrX_{2m,2n}^\iota},\ \Lambda_{\scrX_{2m,2n}^{\iota\vee}}$ and
    $\Lambda_{'\!\scrX_{2m,2n}^{\iota\vee}}$ are Lagrangian subvarieties. The number of their irreducible
      components is equal to the number of centrally symmetric $2m\times2n$-mazes.
\end{prop}

The proof will be given in~\S\ref{slice osp}.

\subsection{Linear side}
In order to unburden the notation, within this particular subsection we denote by $\scrX^{\iota\vee}$ the
symplectic vector space equal to $\Hom(\BC^{2l}_-,\BC^k_+)$, where $\BC^{2l}_-$ (resp.\ $\BC^k_+$)
is equipped with the nondegenerate skew symmetric (resp.\ symmetric) bilinear form as
in~\S\ref{osp setup}. It is equipped with a hamiltonian action of $\Sp(2l)\times\SO(k)$.
The upper triangular Borel subgroup of $\GL(2l)\times\GL(k)$ is denoted $B^\vee$, and the
upper triangular Borel subgroup of $\Sp(2l)\times\SO(k)$ is denoted $B^{\iota\vee}$.

Note that $\scrX^{\iota\vee}$ is the fixed point set of the involution
$\iota$ of $\scrX^\vee:=T^*\Hom(\BC^{2l}_-,\BC^k_+)$ taking a pair
$(C\in\Hom(\BC^{2l}_-,\BC^k_+),\ D\in\Hom(\BC^k_+,\BC^{2l}_-))$ to $(D^*,-C^*)$ (adjoint operators).
Recall that in~\S\ref{linear gl} we identified $\scrX^\vee$ with $T^*\scrQ$ for a quotient vector
space $\scrX^\vee\twoheadrightarrow\scrQ$. The involution $\iota$ induces the same named involution
on $\scrQ$, and the fixed point set $\scrQ^\iota$ is identified with $\Hom(\BC^l_-,\BC^k_+)$,
where $\BC^l_-\subset\BC^{2l}_-$ is spanned by $v_1^-,\ldots,v_l^-$. The action of $B^\vee$ on $\scrQ$
restricts to the action of $B^{\iota\vee}$ on $\scrQ^\iota$. 

As in~Lemma~\ref{relev}, the irreducible components of $\Lambda_{\scrX^{\iota\vee}}$ are in natural
bijection with the relevant $B^{\iota\vee}$-orbits in $\scrQ^\iota$. The involution $\iota$ induces
the same named involution on the set of relevant $B^\vee$-orbits in $\scrQ$.

\begin{lem}
  Each $\iota$-invariant relevant $B^\vee$-orbit in $\scrQ$ contains exactly one relevant
  $B^{\iota\vee}$-orbit in $\scrQ^\iota$.
\end{lem}

\begin{proof}
Recall the construction of representatives of relevant $B^\vee$-orbits in~Lemma~\ref{releva}.
Given a centrally symmetric maze $\fm$ we construct a pair of matrices $(C_\fm,{}^t\!D_\fm)\in\scrQ$;
all their matrix elements are either~0 or~1. The point
$(C_\fm,{}^t\!D_\fm)\in\BO_\fm=B^\vee\cdot(C_\fm,{}^t\!D_\fm)$ is $\iota$-invariant, hence
$(\BO_\fm)^\iota$ is non-empty. It remains to prove that $(\BO_\fm)^\iota$ is a single $B^{\iota\vee}$-orbit.
Since $B^\vee$ (resp.\ $B^{\iota\vee}$) is a semidirect product of its diagonal torus $T^\vee$ (resp.\
$T^{\iota\vee}$) and the unipotent radical, it suffices to prove that $(T^\vee\cdot(C_\fm,{}^t\!D_\fm))^\iota$
is a single $T^{\iota\vee}$-orbit. Equivalently, it suffices to prove that
$(T^\vee\cdot(C_\fm,{}^t\!D_\fm))^\iota$ is connected. But $T^\vee\cdot(C_\fm,{}^t\!D_\fm)$
is open in a vector subspace $V_\fm\subset\scrQ$ (given by linear equations that certain matrix elements
vanish and open conditions that the remaining matrix elements are nonzero). Thus
$(T^\vee\cdot(C_\fm,{}^t\!D_\fm))^\iota$ is open in $V_\fm^\iota$, and hence connected.
\end{proof}

\begin{cor}
  \label{linear osp count}
  The cardinality of $\on{Irr}\Lambda_{\scrX^{\iota\vee}}$ equals the number of the centrally symmetric
  mazes $\sharp \fM^\iota_{2l,k}=\sharp \fM^\iota_{2l+1,k}$. In case $k$ is even, it also equals
  $\sharp \fM^\iota_{2l,k+1}$.
\end{cor}

\begin{proof}
Under the bijection of~Lemma~\ref{releva}(c), the $\iota$-invariant relevant $B^\vee$-orbits
in $\scrQ$ correspond to the centrally symmetric mazes. It remains to apply~Lemma~\ref{symmetric mazes}.
\end{proof}

\subsection{Equivariant slice}
\label{slice osp}
We are in a position to prove~Proposition~\ref{osp lagr}.

\subsubsection{(a)}
\label{(a)} Recall the setup of~\S\ref{slice GL'}. We set $N=2n+1,\ M=2m,\ r=s=n-m,\ k=n+1$.
We have to count the relevant $(H^M_{N,r,s,k})^\iota$-orbits in the $\SO(N)$ flag variety $\CB_N^\iota$.
Indeed, $H^M_{N,r,s,k}=B_M\ltimes\on{U}^{(r,s,k)}_{M,N}$, and
$(H^M_{N,r,s,k})^\iota=B_M^\iota\ltimes(\on{U}^{(r,s,k)}_{M,N})^\iota$, but $M$ is even, and the fixed
point set $B_M^\iota$ coincides with the upper triangular Borel subgroup in $\SO(M)$.
Any relevant $(H^M_{N,r,s,k})^\iota$-orbit is
contained in an $\iota$-invariant relevant $H^M_{N,r,s,k}$-orbit in $\CB_N$. The latter orbits are
indexed by the centrally symmetric $M\times N$-mazes by~Corollary~\ref{iota slice}.

We claim that any such $\iota$-invariant $H^M_{N,r,s,k}$-orbit $\BO_\fm$ in $\CB_N$ contains an
$\iota$-fixed point $p_\fm$. Indeed, $\iota$ gives rise to an action
$\BZ/2\BZ\circlearrowright H^M_{N,r,s,k}$, and it suffices to check $H^1(\BZ/2\BZ,H^M_{N,r,s,k})=\{\cdot\}$.
We have an exact sequence $1\to U\to H^M_{N,r,s,k}\to\BG_m^M\to1$ for a unipotent group $U$.
Since $H^1(\BZ/2\BZ,\BG_a)=0$, we conclude by induction that $H^1(\BZ/2\BZ,U)=\{\cdot\}$, and
it suffices to check that $H^1(\BZ/2\BZ,\BG_m^M)=\{1\}$. The action of the involution is
$\iota(x_1,\ldots,x_{2m})=(x_{2m}^{-1},\ldots,x_1^{-1})$, and the desired triviality of cohomology
follows immediately.

So given a centrally symmetric maze $\fm\in\fM_{M,N}^\iota$, the orbit
$\BO_\fm=H^M_{N,r,s,k}\cdot p_\fm\subset\CB_N$ is
equal to $H^M_{N,r,s,k}/\on{Stab}_{H^M_{N,r,s,k}}p_\fm$. Hence the fixed point set $\BO_\fm^\iota$ is a
single orbit of $(H^M_{N,r,s,k})^\iota$.

We conclude that the relevant $(H^M_{N,r,s,k})^\iota$-orbits in $\CB_N^\iota$ are indexed by the
set $\fM_{M,N}^\iota$ of centrally symmetric $M\times N$-mazes. It remains to
apply~Corollary~\ref{linear osp count}.

\subsubsection{(b)} We recall the setup of~\S\ref{slice GL''} and set $N=2n,\ M=2m+1,\ r=s=n-m-1$.
We have to count the relevant $H^{M\iota}_{N,\on{mid}}$-orbits in the $\SO(N)$ flag variety $\CB_N^\iota$.
Here $H^{M\iota}_{N,\on{mid}}\subset\SO(N)$ is the neutral connected component of the fixed point set
$(H^M_{N,\on{mid}})^\iota$.
Arguing as in~\S\ref{(a)}, for a centrally symmetric maze $\fm\in\fM_{M,N}^\iota$, we find an
$\iota$-fixed point $p_\fm$ in the corresponding $H^M_{N,\on{mid}}$-orbit $\BO_\fm$ in $\CB_N$.
To this end it suffices to check $H^1(\BZ/2\BZ,H^M_{N,\on{mid}})=\{\cdot\}$ that again boils down
to $H^1(\BZ/2\BZ,\BG_m^M)=\{1\}$. This time the action of involution is
$\iota(x_1,\ldots,x_{m+1},\ldots,x_{2m+1})=(x_{2m+1}^{-1},\ldots,x_{m+1}^{-1},\ldots,x_1^{-1})$, but again
it is straightforward to check the desired triviality of cohomology.

Now the argument of~\S\ref{(a)} runs into~2 problems in the present situation.
First, the fixed point set $(B_M)^\iota$ has~2 connected components: the neutral one is the upper
triangular Borel subgroup $B_M^\iota\subset\SO(M)$, and the other one is $-1_M\cdot B_M^\iota$.
Second, the fixed point set $(\CB_N)^\iota$ has~2 connected components: the neutral one
(the image of $\SO(N)\subset\on{O}(N)=\GL(N)^\iota$) is the flag variety $\CB_N^\iota$ of $\SO(N)$.
It follows that the fixed point set $\BO_\fm^\iota$ has two connected components: one is a single
$H^{M\iota}_{N,\on{mid}}$-orbit in $\CB_N^\iota$, the other one is a single $H^{M\iota}_{N,\on{mid}}$-orbit
in the other connected component of $(\CB_N)^\iota$, and together they form a single
$(H^M_{N,\on{mid}})^\iota$-orbit.

\subsubsection{(c)} In case $m<n$, we set $M=2m,\ N=2n$, and recall the setup of~\S\ref{slice GL'''}.
As in~\S\ref{(a)} we have to count the relevant $(H^M_{N,\on{mid}})^\iota$-orbits in the $\Sp(N)$ flag
variety $\CB_N^\iota$. The argument similar to the one in~\S\ref{(a)} implies that the relevant
$(H^M_{N,\on{mid}})^\iota$-orbits in the flag variety $\CB_N^\iota$ of $\Sp(N)$ are indexed by the
set $\fM_{M,N}^\iota$ of centrally symmetric $M\times N$-mazes. It remains to
apply~Corollary~\ref{linear osp count}.

In case $m=n$, we set $N=2n$. We set
$\BC^n=\BC v_1\oplus\ldots\oplus\BC v_n\subset\BC^N$. Then $\scrX^\iota_{2n,2n}$ is isomorphic to
the cotangent bundle $T^*\scrY^\iota_{2n,2n}$, where $\scrY^\iota_{2n,2n}=\Sp(2n)\times\BC^n$. Note that
the upper triangular Borel subgroup $B^\iota=B_N^\iota\times B_N^\iota\subset\Sp(2n)\times\Sp(2n)$ acts on
$\scrY^\iota_{2n,2n}$, but the isomorphism $\scrX^\iota_{2n,2n}\simeq T^*\scrY^\iota_{2n,2n}$ is {\em not}
$B^\iota$-equivariant. We have to count the relevant $B^\iota$-orbits in $\scrY^\iota_{2n,2n}$.
The argument similar to the one in~\S\ref{(a)}, together with~Corollary~\ref{symp mirab} implies
that these orbits are indexed by the centrally symmetric $N\times N$-mazes. It remains to
apply~Corollary~\ref{linear osp count}.

\end{document}